\documentclass[11pt,a4paper]{article}
\usepackage{float}
\usepackage{amssymb,amsmath,amsfonts,mathrsfs,bm,enumerate,color,bbm,url}
\usepackage{mathtools}
\usepackage{times,theorem,latexsym,color,comment}
\allowdisplaybreaks[1]
\numberwithin{equation}{section}

\usepackage[colorlinks=true, pdfstartview=FitV, linkcolor=blue, citecolor=blue, urlcolor=blue,pagebackref=false]{hyperref}

\usepackage{tocloft}

\usepackage{tikz}
\usepackage{float}
\usepackage[font=footnotesize]{caption}

\usepackage{color} 
\newcommand{\BOX}{\ensuremath\Box}

\newtheorem{theorem}{Theorem}[section]

{\theorembodyfont{\rmfamily}}
{\theorembodyfont{\rmfamily}}
{\theorembodyfont{\rmfamily}}
\newtheorem{lemma}[theorem]{Lemma}
\newtheorem{proposition}[theorem]{Proposition}
{\theorembodyfont{\rmfamily}\newtheorem{remark}[theorem]{Remark}}
{\theorembodyfont{\rmfamily}}

\newcommand{\N}{\mathbb{N}}
\newcommand{\Z}{\mathbb{Z}}
\newcommand{\R}{\mathbb{R}}

\renewcommand{\S}{\mathbb{S}}
\newcommand{\T}{\mathbb{T}}

\newcommand{\dd}{\,{\rm d}}
\newcommand{\opsupp}{\operatorname{supp}}
\newcommand{\opint}{\operatorname{int}}

\newcommand{\opbog}{\operatorname{Bog}}
\newcommand{\opbogeps}{\operatorname{\mathcal{B}^\eps}}

\newcommand{\opdist}{\operatorname{dist}}
\newcommand{\eps}{\varepsilon}

\def\XXint#1#2#3{{\setbox0=\hbox{$#1{#2#3}{\int}$}
		\vcenter{\hbox{$#2#3$}}\kern-.5\wd0}}

\newenvironment{proof}{{\vskip\baselineskip\noindent\textbf{Proof:}}}
{\hspace*{.1pt}\hspace*{\fill}\BOX\vskip\baselineskip}

\newenvironment{proofx}[1]
{\vskip\baselineskip\noindent\textbf{Proof of {#1}:}}
{\hspace*{.1pt}\hspace*{\fill}\BOX\vskip\baselineskip}

\definecolor{darkgreen}{rgb}{0,0.5,0}
\definecolor{darkblue}{rgb}{0,0,0.7}
\definecolor{darkred}{rgb}{0.9,0.1,0.1}

\begin{document}

\title{Lagrangian controllability in perforated domains}

\author{
Mitsuo Higaki
\thanks{
Department of Mathematics, 
Graduate School of Science, 
Kobe University, 
1-1 Rokkodai, Nada-ku, Kobe 657-8501, Japan.
\textit{E-mail address:}\texttt{higaki@math.kobe-u.ac.jp}
}
\and
Jiajiang Liao
\thanks{
School of Mathematical Sciences, 
Beihang University, 
102206 Beijing, China.
\textit{E-mail address:}\texttt{jjliao@buaa.edu.cn}
}
\and
Franck Sueur
\thanks{
Department of Mathematics, 
Maison du nombre, 6 avenue de la Fonte, 
University of Luxembourg,  
L-4364 Esch-sur-Alzette, Luxembourg.
\textit{E-mail address:}\texttt{Franck.Sueur@uni.lu}
}
}
\date{}

    \maketitle

\begin{abstract} 
The question at stake in Lagrangian controllability is whether one can move a patch of fluid particles to a target location by means of remote action in a given time interval. In the last two decades, positive results have been obtained both for the incompressible Euler and Navier-Stokes equations. However, for the latter, the case where the fluid is contained within domains bounded by solid boundaries with the no-slip condition has not been addressed, with respect to the difficulty caused by viscous boundary layers. In this paper, we investigate the Lagrangian controllability of viscous incompressible fluid in perforated domains for which the fraction of volume occupied by the holes is sufficiently small. Moreover, we quantitatively distinguish situations depending on the parameters for holes (diameter and distance) and for fluid (size of the initial data). Our approach relies on recent results on homogenization for evolutionary problems and on weak-strong stability estimates in measure of flows, alongside classical results on Runge-type approximations for elliptic equations and on Cauchy-Kowalevsky-type theorems for equations with analytic coefficients. Here, homogenization refers to the vanishing viscosity limit outside a porous medium, where (after scaling in time) the Navier-Stokes equations are homogenized to the Euler or Darcy equations. Indeed, in the proof, we act on the Navier-Stokes equations by strong and fast forcing to leverage inviscid approximations, which is a standard technique in the theory of controllability.
\end{abstract}

    \tableofcontents

    \section{Introduction}
    \label{sec.intro}

A classical problem in the controllability of the partial differential equations (PDEs) is determining if it is feasible to navigate the state of a physical system described by a PDE to a desired or specified target state within an imparted time. For instance, if one considers the incompressible Euler or Navier-Stokes equations, the problem is to ask if the state of the fluid velocity field can be modified; see in particular \cite{Lio88,Cor93,Cor9596,Gla00,Cor07,CMS20}.

On the other hand, a more recent problem under consideration is the Lagrangian controllability. This problem seeks to determine whether it is feasible to navigate a patch of fluid particles to a target location within an imparted time by means of a remote action, typically imposed on a portion of the boundaries surrounding the fluid. In these last two decades, positive results for the Lagrangian controllability have been obtained both for the incompressible Euler and Navier-Stokes equations; see \cite{Hor08,GlaHor10,GlaHor12,GlaHor16,HorKav17,LSZ22a}.

Let us briefly summarize the known results for the Lagrangian controllability of the Navier-Stokes equations with macroscopic bodies with solid boundaries. For the case of Navier-slip conditions on the boundaries, substantial progress has been achieved in \cite{LSZ22a}; see also \cite{LSZ25} studying the MHD equations. However, for the case of the no-slip boundary condition, no results have been reported even in the analytic framework.

A significant challenge arising from the no-slip boundary condition is the considerable difficulty in estimating the impact of viscous boundary layers near boundaries. The standard method for addressing controllability, as in \cite{CMSZ19,CMS20,LSZ22a,LSZ22b}, involves time scaling to introduce small viscosity in the Navier-Stokes equations, thereby leveraging the controllability of the inviscid equations. The underlying idea is to act on the viscous equations by a strong and fast control to bring them closer to the inviscid counterparts. Therefore, the method requires studying the vanishing viscosity limit of the Navier-Stokes equations, a problem recognized as notably difficult when the no-slip conditions are imposed on macroscopic boundaries; see, for example, \cite{MaeMaz18}.

In this paper, we consider the Lagrangian controllability for non-analytic boundaries with the no-slip condition, focusing on the case when boundaries are small. In fact, the limit of vanishing viscosity has been studied outside a porous medium, that is, in a domain perforated by holes with no-slip condition whose proportion in the total fluid is sufficiently small relative to viscosity \cite{ILFNL09,LacMaz16,Hof23}. In these situations, the local Reynolds number is bounded and thus the viscous boundary layers are negligible, enabling us to perform a simple energy computation similar to that in \cite{Kato84} and to avoid the Prandtl boundary layers. A prototypical result \cite[Theorem 1]{ILFNL09} says that a Leray-Hopf weak solution of the Navier-Stokes equations with viscosity $\nu$ converges to the solution to the Euler equations in the limit $\nu\to0$ when the fluid domain is the exterior to a single body whose volume is sufficiently smaller than $\nu$. This result is extended by \cite{LacMaz16} quantitatively by introducing parameters for holes (diameter and distance) and for fluid (viscosity and size of the initial data), and by \cite{Hof23} to the critical and supercritical cases of parameters where the limit equations are shown to be the Euler-Brinkman and Darcy equations, respectively, contrasting with the Euler equations in the subcritical case. The quantitative estimates in \cite{LacMaz16,Hof23} are pivotal when considering the Lagrangian controllability in a perforated domain.

This paper is organized as follows. In Section \ref{sec.LC.E}, we present a problem of the Lagrangian controllability in partially perforated domains, and prove the main result Theorem \ref{thm.main.E}. In Section \ref{sec.LC.D}, we present a problem of the Lagrangian controllability in fully perforated domains, and prove the main result Theorem \ref{thm.main.D}. In Appendix \ref{appx.Defs}, we collect several notions from differential geometry. In Appendix \ref{appx.Flow}, we recall the notion of a flow associated with a Leray-type solution of the Navier-Stokes equations. The materials in Appendices \ref{appx.Defs}--\ref{appx.Flow} will be referenced throughout this paper.

    \section{Lagrangian controllability in partially perforated domains}
    \label{sec.LC.E}

Let us consider the 3d Navier-Stokes equations with viscosity $1$ and forcing $F^\eps$
\begin{equation}\label{eq.NS}
    \left\{
    \begin{array}{ll}
    \partial_t U^\eps - \Delta U^\eps + U^\eps\cdot\nabla U^\eps + \nabla P^\eps 
    = F^\eps &\mbox{in}\ (0,T)\times\Omega^\eps,\\
    \nabla\cdot U^\eps = 0 &\mbox{in}\ (0,T)\times\Omega^\eps,\\
    U^\eps = 0 &\mbox{on}\ (0,T)\times\partial\Omega^\eps,\\
    U^\eps = U^\eps_0 &\mbox{on}\ \{0\}\times\Omega^\eps.
    \end{array}\right.
\end{equation}
Assuming $\Omega^\eps$ to be a partially perforated domain, we are interested in the following Lagrangian controllability problem: move fluid particles that are initially in a polluted zone $P_0^\eps$ in a set $K$ toward a safe zone $P_1$ without pollution away from $K$ in some time $T_c>0$. The way to move the particles is by remote action: regard the term $F^\eps$ in \eqref{eq.NS} as a control forcing moving $P_0^\eps$ to $P_1$ and assume that $F^\eps$ is smooth and supported in a control zone $\Omega_c$. The control zone $\Omega_c$ is supposed to be away from both $K$ and $P_1$.

An intrinsic feature of this problem is that one cannot hope to move all the fluid particles, not even the full Lebesgue measure of the patch of particles. This occurs because the particles close to the no-slip boundaries are unable to escape from their, yet shrinking, connected component. Hence, we will focus on the goal of moving most of the particles.

The precise definitions of $\Omega^\eps, P_0^\eps, P_1$, and $\Omega_c$ are as follows. In this paper, we regard the $3$-torus $\T^3$ as the cube $[0,1]^3$ with identified sides. It is equivalent to regard $\T^3$ as the orbit space of $\R^3$ under the action of translations by $\Z^3$. Then we can equip $\T^3$ with a Riemannian metric $g$ for which the quotient mapping $\pi: \R^3 \to \T^3$ is a local isometry. This remark is useful when lifting up notions defined in subspaces of $\R^3$ to $\T^3$; see Appendix \ref{appx.Defs}.

\begin{itemize} 
\item
\textbf{Definition of a partially perforated domain $\Omega^\eps$} \\
Let $K$ be a closed set in $\T^3$. To avoid technicality, we assume that $K$ is a closed cube with length $0<L<1$. Take $\alpha>3/2$ and sufficiently large $N\in\N$, and decompose $K$ as 
$K = \bigcup_i \overline{\mathcal{Q}^\eps_i}$ with $N^3$ small open disjoint cubes $\mathcal{Q}^\eps_i$ with center $x^\eps_i$ and sidelength 
\begin{equation}\label{def.eps}
    \eps = \frac{L}{N}. 
\end{equation}
Choose a reference particle $\mathcal{T}$ from $\R^3$ such that $\mathcal{T}$ is homeomorphic to the closed three-dimensional unit ball, with smooth boundary, containing $0$ in its interior, and contained in $B_{1/8}(0)$. Then we consider $N^3$ particles $\mathcal{T}^\eps_i = x^\eps_i + \eps^\alpha \mathcal{T} \subset \mathcal{Q}^\eps_i$ and let 
\[
    \Omega^\eps = \T^3 \setminus \bigcup_i \mathcal{T}^\eps_i. 
\]
The set $\Omega^\eps$ obtained in this way is called a partially perforated domain.

\item
\textbf{Definition of a polluted zone $P_0^\eps$ and a safe zone $P_1$} \\
Let each of $P_0$ and $P_1$ be a Jordan domain in $\T^3$, that is, the interior of a Jordan surface in $\T^3$; see Appendix \ref{appx.Defs} for the definition of a Jordan surface and its properties. Let $\gamma_0$ and $\gamma_1$ denote the Jordan surfaces of $P_0$ and $P_1$, respectively. Assume that $\gamma_0$ and $\gamma_1$ are isotopic in $\T^3$ and surround the same volume. We then assume that $P_0 \subset K$ and that $P_1\cap K = \emptyset$. Finally, we set $P_0^\eps = P_0\cap \Omega^\eps$.

\item
\textbf{Definition of a control zone $\Omega_c$} \\
Let $\Omega_c\neq \emptyset$ be an open set in $\T^3$ such that $\overline{\Omega_c}\cap K = \overline{\Omega_c}\cap P_1 = \emptyset$. Notice that $\gamma_0$ and $\gamma_1$ above are isotopic in $\T^3\setminus\overline{\Omega_{c}}$. Since one can always pick a sufficiently small ball with a smooth boundary in $\Omega_c$, to avoid complicating descriptions, we may assume that $\Omega_c$ itself is a ball with a smooth boundary. 
\end{itemize}

\begin{figure}[H]
        \centering
        {\includegraphics[scale=0.86]{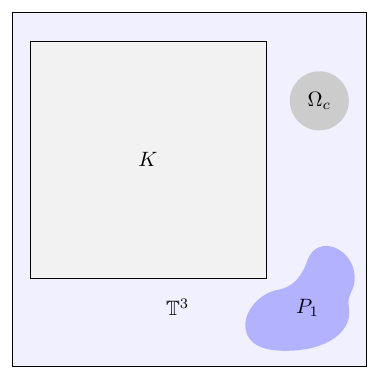}}
        \qquad   \qquad   
        {\includegraphics[scale=0.65]{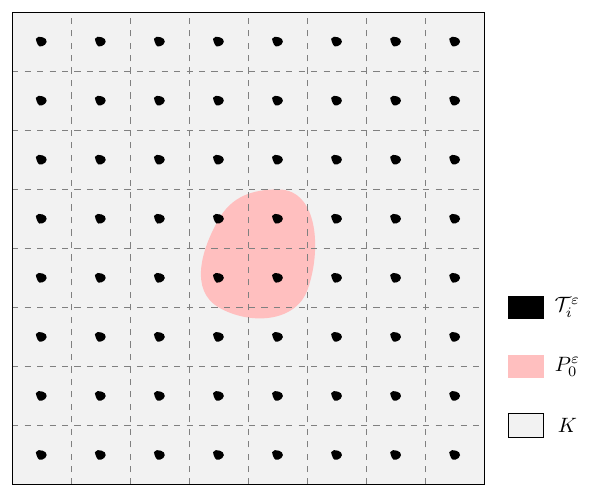}}
         \caption{on the left: $K$, $\Omega_c$ and $P_1$; on the right: zoom on $K$, with $P_0^\eps$ and $(\mathcal{T}^\eps_i)_i$}
\end{figure}

Then we consider the Leray-Hopf weak solutions of \eqref{eq.NS} which satisfy the energy inequality. The existence of a Leray-Hopf weak solution is elementary and can be proved by the Galerkin method, e.g., in Temam \cite[Section 3.1, Chapter $\mathrm{I}\hspace{-1.2pt}\mathrm{I}\hspace{-1.2pt}\mathrm{I}$]{Tem84book}.

Now let us state a result of the Lagrangian controllability for Leray-Hopf weak solutions of the Navier-Stokes equations \eqref{eq.NS}. Theorem \ref{thm.main.E} below says that in a nutshell, one can achieve the Lagrangian approximate controllability in a short time for initial data that are not necessarily small if the subparameter $\beta$ in the statement is suitably chosen. It is worth highlighting that Theorem \ref{thm.main.E} does not require the uniqueness of Leray-Hopf weak solutions; see also Remark \ref{rem.thm.main.E} (\ref{item2.rem.thm.main.E}).
Denote by $L^2_\sigma(\Omega^\eps)$ the subspace of $L^2(\Omega^\eps)^3$ whose elements are divergence-free in the sense of distributions and have a vanishing normal trace in $\partial\Omega^\eps$. Define 
\begin{equation}\label{def.pE}
    \mathfrak p_E 
    = \mathfrak p_E(\alpha,\beta)
    = \min\Big\{\alpha+\beta-3,\, \alpha-\frac{3}{2},\, \beta\Big\}.
\end{equation}

\begin{theorem}\label{thm.main.E}
Let $\alpha>3/2$ and choose $\beta>0$ so that $3-\alpha < \beta < \alpha$, which ensures that $\mathfrak p_E$ in \eqref{def.pE} is positive. For given $0<\eta<1$, we define
\begin{equation}\label{def.eps_E}
     \eps_0=\frac{L}{\lfloor \eta^{-2/\mathfrak{p}_E}\rfloor+1},
\end{equation}
where $0<L<1$ refers to the length of $K$, $\lfloor\,\cdot\,\rfloor$ the floor function. Then, for any $0 < \eps \le \eps_0$ of the form \eqref{def.eps} and initial data $U^\eps_0\in L^2_\sigma(\Omega^\eps)$ satisfying
\begin{equation}\label{est1.thm.main.E}
    \|U^\eps_0\|_{L^2(\Omega^\eps)}
    \le
    \eps^{\mathfrak p_E - \beta}, 
\end{equation}
the Lagrangian controllability of \eqref{eq.NS} holds within the time $\eps^\beta$: there exists a smooth control forcing $F^\eps$ supported in $\Omega_c$ with $\|F^\eps\|_{L^{\infty}_t H^s_x} \le C_0 \eps^{-2\beta}$ for any $s\in\N$ 
and a constant $C_0=C_0(s)>0$ depending on $s$
such that
\begin{equation}\label{est2.thm.main.E}
    \bigcup_{t\in [0,\eps^{\beta}]} \Phi^\eps(t,P_0^\eps) \subset\Omega^\eps\setminus\overline{\Omega_{c}}, \qquad
    \mathcal{L}(P_1 \setminus \Phi^\eps(\eps^\beta,P_0^\eps)) 
    \le C \eta, 
\end{equation}
where $\mathcal{L}$ refers to the Lebesgue measure and $\Phi^\eps$ the flow for a corresponding Leray-Hopf weak solution $U^\eps$ to \eqref{eq.NS}. The constant $C>0$ is independent of $\eta$ and $\eps$. 
\end{theorem}

\begin{figure}[H]
        \centering
        {\includegraphics[scale=0.7]{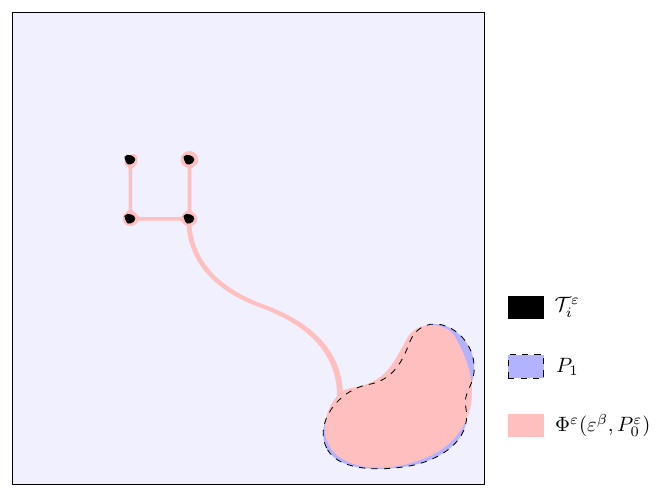}}
         \caption{position at time $\eps^\beta$ of the particles initially in $P_0^\eps$}
\end{figure}

\begin{remark}\label{rem.thm.main.E}
\begin{enumerate}[(i)]
\item\label{item2.rem.thm.main.E}
We emphasize the possibility that a Leray-Hopf weak solution $U^\eps$ in Theorem \ref{thm.main.E} may not be unique, since no smallness or structural conditions are imposed on the force $F^\eps$. However, the statement of Theorem \ref{thm.main.E} holds true regardless of the selection of the solutions. It should be noted that, for any selection of $U^\eps$, the flow $\Phi^\eps$ is uniquely determined almost everywhere due to the regularity of $U^\eps$; see Appendix \ref{appx.Flow}.

\item\label{item3.rem.thm.main.E}
We point out that the limit as $\eps$ goes to $0$ in Theorem \ref{thm.main.E} is a singular limit. This singularity is due to the scaling of $(U^\eps, P^\eps, F^\eps)$ in \eqref{eq.UPF_E} and not to the appearance of boundary layers near the boundaries of $\Omega^\eps$, triggered by large forcing $F^\eps$. In fact, boundary effects are limited due to the small volume occupied by the holes in $\Omega^\eps$, and can be treated after scaling \eqref{eq.UPF_E} by homogenization as in \cite{Hof23}.

\item\label{item5.rem.thm.main.E}
The question of ``the best choice" of the free parameter $\beta$ is interesting when optimizing the trio of cost (the size of control forcings), time, and the size of initial data. 
However, note that there should be different relevant choices of $\beta$ depending on which of the three characteristics we are focusing on.  
\end{enumerate}
\end{remark}

    \subsubsection*{Outlined proof of Theorem \ref{thm.main.E}.}

We outline the proof of Theorem \ref{thm.main.E} by dividing it into four steps. The order in which the contents are described in the steps differs from that of this paper. This is to highlight ideas in each step more conceptually, which may have an advantage for future extension to isolate the blocks of argument; any improvement such as higher-order asymptotics or norms in the homogenization in Step 2 should allow us to try other types of stability in Step 3.

\begin{enumerate}[Step 1.]
\item
\textbf{Scaling.} 
Assume that the control forcing $F^\eps$ is obtained. Motivated by the vanishing viscosity limit in $\Omega^\eps$ in Step 2 below, we perform a scaling in time 
\begin{equation}\label{eq.UPF_E}
\begin{split}
    U^\eps(t,x) &= \eps^{-\beta} u^\eps(\eps^{-\beta} t,x), \\
    P^\eps(t,x) &= \eps^{-2\beta} p^\eps(\eps^{-\beta} t,x), \\
    F^\eps(t,x) &= \eps^{-2\beta} f^{\eps}(\eps^{-\beta} t,x), 
\end{split}
\end{equation}
and consider \eqref{eq.NS_fc_E} instead of \eqref{eq.NS}, that is, the Navier-Stokes equations with viscosity $\eps^\beta$ instead of the ones with viscosity $1$. Since the flows $\Phi^\eps$ and $\phi^\eps$ for $U^\eps$ and $u^\eps$, respectively, are in relation \eqref{eq.FlowMaps_E}, it is advantageous to work on the Leray-Hopf weak solutions $u^\eps$ solving \eqref{eq.NS_fc_E} during the proof of Theorem \ref{thm.main.E}.

\item
\textbf{Homogenization.} 
We approximate Leray-Hopf weak solutions $u^\eps$ of \eqref{eq.NS_fc_E} by a solution $u_E$ of the Euler equations in the torus $\T^3$. Here, $u_E$ is constructed in Step 4 below. The proof of the convergence rate in $L^\infty_t L^2_x$ is inspired by quantitative analysis in H\"{o}fer \cite{Hof23}. However, unlike \cite{Hof23}, the perforation in $\Omega^\eps$ is localized in a subset $K$ in our problem. Thus, we construct new correctors in partially perforated domains $\Omega^\eps$ by revisiting \cite{Hof23} and Allaire \cite{All90a,All90b}. Then, following \cite[Proof of Theorem 1.2]{Hof23}, we obtain the convergence rate between $u^\eps$ and $u_E$ in $L^\infty_t L^2_x$.

\item 
\textbf{Stability.} 
By a weak-strong stability argument for flows, we convert the convergence rate between $u^\eps$ and $u_E$ obtained in Step 2 to the estimate of the difference between flows $\phi^\eps$ and $\phi_E$, where $\phi_E$ refers to the flow for $u_E$.

\item
\textbf{Limit equations.} 
We construct a smooth triplet $(u_E, p_E, f_E)$ solving the Euler equations in $\T^3$ such that the initial zone $P_0$ moves by $u_E$ to the zone $P_1$, up to an arbitrarily small error in measure. Moreover, we take $f^{\eps}$ by $f^{\eps} := f_E|_{\Omega^\eps}$ and define the control forcing $F^\eps$ through \eqref{eq.UPF_E}. Our construction of $(u_E, p_E, f_E)$ is inspired by Glass-Horsin \cite{GlaHor12} for the controllability of the 3d Euler equations in smooth bounded domains. We adapt the argument in \cite{GlaHor12} to the case of the torus $\T^3$. It is emphasized that the precise value or structure of $f_E$ is not important in the whole proof; only the uniformity of its bounds on its norm counts.
\end{enumerate}

\begin{remark}\label{rem.outline}
The corresponding sections in this paper are as follows: Section \ref{sec.Limit.E} for Step 4,  Section \ref{sec.Homogenization.E} for Step 2, and Section \ref{sec.Scaling-Stability.E} for Step 1 and Step 3. 
\end{remark}

Before proceeding with the proof of Theorem \ref{thm.main.E}, which is the Lagrangian controllability by an Euler homogenization, let us discuss the controllability by an Euler-Brinkman homogenization. The Euler-Brinkman equations in $\T^3$ are written as 
\begin{equation}\label{eq.EB}
    \left\{
    \begin{array}{ll}
    \partial_t u_{EB} + u_{EB}\cdot\nabla u_{EB} + \mathcal{R} \mathbbm{1}_{K} u_{EB} + \nabla p_{EB} = f_{EB} &\mbox{in}\ (0,1)\times\T^3,\\
    \nabla\cdot u_{EB} = 0 &\mbox{in}\ (0,1)\times\T^3,\\
    u_{EB}=0 &\mbox{on}\ \{0\}\times\T^3, 
    \end{array}\right.
\end{equation}
where $\mathcal{R}$ refers to the resistance matrix in \eqref{resis.mat} and $\mathbbm{1}_{K}$ to the indicator function of $K$ in which the holes are located. These equations appear as limit equations in Steps 2 and 4 when the parameters $\alpha,\beta$ satisfy $3/2 < \alpha < 3$ and $\beta = 3 - \alpha$; see \cite[Theorem 1.1]{Hof23}.

In Step 1 of the proof, one just needs to modify the scaling \eqref{eq.UPF_E} according to the homogenization process in Step 2. In Step 2, one can obtain the convergence rate between $u^\eps$ and $u_{EB}$ following \cite[Proof of Theorem 1.1]{Hof23}, that is, Proposition \ref{prop.RoC.E} with $u_E$ replaced by $u_{EB}$ and $p_E$ by the exponent $\mathfrak p_{EB} = \min\{\alpha-3/2, 3-\alpha\}$. Step 3 is the same as the one for Theorem \ref{thm.main.E}. Hence, new considerations are required only in Step 4, which are for the controllability of the Euler-Brinkman equations. However, even in the case where the matrix $ \mathcal{R}$ is a homothety, the presence of $\mathbbm{1}_{K}$ prevents us from making direct use of (forced) potential flows, as in Proposition \ref{prop.CoE} below. Note that, if one allows the time intervals for the controllability shorter in algebraic order of $\eps$ than that in Theorem 1.1, one may try another time and amplitude rescaling to \eqref{eq.EB}, such as in \eqref{eq.UPF_E}, to recast the system as a perturbation of the Euler equations by a lower-order perturbation due to the Brinkman term, to rely on the controllability by an Euler homogenization. We will not explore both directions further in this paper.

    \paragraph{Notation}

In the remainder of this section, we will adopt the following notation.
\begin{itemize}
\item
We write $A\lesssim B$ if there is some constant $C > 0$, called the implicit constant, such that $A\le CB$. We also write $A\approx B$ if there is some implicit constant $C\ge1$ such that $C^{-1}A\le B\le CA$. {\it Unless stated, any implicit constant $C$ is independent of $\eps$.}

\item 
For $x,y\in \T^3$, let us agree that $\opdist(x,y)$ denotes the distance between $x$ and $y$ if there is no ambiguity. For $x\in\T^3$ and $A,B\subset \T^3$, in the same manner, $\opdist(x,A)$ and $\opdist(A,B)$ denote the distances between $x$ and $A$, and $A$ and $B$, respectively.

\item 
For $A\subset \T^3$ and $\xi > 0$, as long as there is no ambiguity, we let 
\[
    \mathcal{V}_{\xi}[A] = \{x\in \T^3~|~ \opdist(x,A) < \xi\}.
\]
\end{itemize}

    \subsection{Controllability of Euler equations in torus}
    \label{sec.Limit.E}

In this section, we prove the remote controllability of the Euler equations in $\T^3$. For notions such as Jordan surface and isotopicity, we refer to Appendix \ref{appx.Defs}.

\begin{proposition}\label{prop.CoE}
Assume that Jordan surfaces $\gamma_0$ and $\gamma_1$ are isotopic in $\T^3\setminus\overline{\Omega_{c}}$ and surround the same volume. For any $\eta>0$ and $k\in\N$, there exist smooth $(u_E, p_E, f_E)$ and $d_0>0$ satisfying 
\begin{equation}\label{eq1.prop.CoE}
    \left\{
    \begin{array}{ll}
    \partial_t u_E + u_E\cdot\nabla u_E + \nabla p_E = f_E &\mbox{in}\ (0,1)\times\T^3,\\
    \nabla\cdot u_E = 0 &\mbox{in}\ (0,1)\times\T^3,\\
    u_E=0 &\mbox{on}\ \{0\}\times\T^3,
    \end{array}\right.
\end{equation}
and the conditions 
\begin{gather}
    \opsupp f_E(t,\cdot) \subset \overline{\Omega_c}, 
    \quad t\in(0,1); \label{item1.prop.CoE} \\
    \opdist(\phi_E(t,\gamma_0),\Omega_c)\geq d_0, 
    \quad t\in [0,1]; \label{item2.prop.CoE} \\
    \|\phi_E(1,\gamma_0)-\gamma_1\|_{C^k(\S^2)}<\eta,  \label{item3.prop.CoE}
\end{gather}
up to reparameterization. Here, $\phi_E$ denotes the flow of $u_E$. 
\end{proposition}

In fact, this proposition is an adaptation of \cite[Theorem 1.5]{GlaHor12} stated for bounded domains in $\R^3$. Nevertheless, we will provide the proof with a certain degree of detail for the sake of completeness, emphasizing the differences due to our periodic setting.

In the proof, we need Lemma \ref{lem.PotFlows.E} below, whose proof is postponed for a moment.

\begin{lemma}\label{lem.PotFlows.E}
Assume that Jordan surfaces $\gamma_0$ and $\gamma_1$ are isotopic in $\T^3\setminus\overline{\Omega_{c}}$ and surround the same volume. For any $\eta>0$ and $k\in\N$, there exist smooth $\theta$ supported in $(0,1)\times\T^3$ and $d_0>0$ satisfying the conditions 
\begin{gather}
    \opsupp \Delta \theta(t,\cdot) \subset \overline{\Omega_c}, \quad t\in(0,1); \label{item1.lem.PotFlows.E} \\
    \opdist(\phi^{\nabla\theta}(t,\gamma_0),\Omega_c)\geq d_0, \quad t\in[0,1]; \label{item2.lem.PotFlows.E} \\
    \|\phi^{\nabla\theta}(1,\gamma_0)-\gamma_1\|_{C^k(\S^2)}<\eta, \label{item3.lem.PotFlows.E}
\end{gather}
up to reparameterization. Here $\phi^{\nabla \theta}$ denotes the flow of $\nabla\theta$. 
\end{lemma}

\begin{proofx}{Proposition \ref{prop.CoE}}
Let $\eta>0$ and $k\in\N$ be given. Take $\theta$ in Lemma \ref{lem.PotFlows.E} and define
\[
    v_E = \nabla \theta, 
    \qquad
    q_E = -\partial_t\theta-\frac{1}{2}|\nabla\theta|^2,
    \qquad
    g_E = \Delta\theta.
\]
Then, since 
\[
    v_E\cdot\nabla v_E
    = \nabla \Big(\frac{1}{2}|\nabla\theta|^2\Big), 
\]
we see that $(v_E,q_E)$ satisfies the Euler equations with inhomogeneous divergence
\begin{equation}\label{eq1.prf.prop.CoE}
    \left\{
    \begin{array}{ll}
    \partial_t v_E + v_E\cdot\nabla v_E + \nabla q_E = 0 &\mbox{in}\ (0,1)\times\T^3,\\
    \nabla\cdot v_E = g_E &\mbox{in}\ (0,1)\times\T^3,\\
    v_E=0 &\mbox{on}\ \{0\}\times\T^3.
    \end{array}\right.
\end{equation}

We claim that there is a smooth vector field $w_E$ such that
\begin{equation}\label{eq2.prf.prop.CoE}
    \opsupp w_E \subset \overline{\Omega_c},
    \qquad
    \nabla\cdot w_E = g_E.
\end{equation}
For demonstrating the claim, it suffices to show that 
\begin{equation}\label{est1.prf.prop.CoE}
    \int_{\T^3} g_E = \int_{\Omega_c} g_E = 0.
\end{equation}
Indeed, a Bogovskii operator $\opbog=\opbog[\,\cdot\,]$ is applicable to $g_E$ and $w_E:=\opbog[g_E]$ meets \eqref{eq2.prf.prop.CoE}; see, e.g. \cite{BorSoh90} for the properties of $\opbog$. To prove \eqref{est1.prf.prop.CoE}, we use the structure $g_E = \Delta \theta = \nabla\cdot (\nabla \theta)$. Since $\opsupp \Delta \theta \subset \overline{\Omega_c}$, the Gauss divergence theorem gives 
\[
    0 
    = \int_{\T^3\setminus \overline{\Omega_c}} \Delta \theta 
    = -\int_{\partial \Omega_c} (\nabla \theta) \cdot n \dd \sigma, 
\]
where $n$ denotes the outward unit normal vector to $\partial \Omega_c$ and $\dd \sigma$ the surface measure of $\partial \Omega_c$. Then \eqref{est1.prf.prop.CoE} follows from another application of the Gauss divergence theorem 
\[
    \int_{\Omega_c} g_E 
    = \int_{\Omega_c} \Delta \theta 
    = \int_{\partial \Omega_c} (\nabla \theta) \cdot n \dd \sigma
    = 0.
\]

Now define $u_E = v_E - w_E$, $p_E = q_E$, and
\[
    f_E = -\partial_t w_E - v_E\cdot\nabla w_E - w_E\cdot\nabla v_E + w_E\cdot\nabla w_E.
\]
Then $(u_E,p_E,f_E)$ satisfies the properties in Proposition \ref{prop.CoE} thanks to \eqref{eq1.prf.prop.CoE}--\eqref{eq2.prf.prop.CoE}.
\end{proofx}

The rest of this section focuses on proving Lemma \ref{lem.PotFlows.E}. We apply the two lemmas below.

\begin{lemma}\label{lem.KryGlaHor}
Assume that Jordan surfaces $\gamma_0$ and $\gamma_1$ are isotopic in $\T^3\setminus\overline{\Omega_{c}}$ and surround the same volume. Then the following hold.
\begin{enumerate}[(1)]
\item\label{item1.lem.KryGlaHor}
There exists a volume-preserving diffeotopy $h\in C^{\infty}([0,1]\times\T^3;\T^3)$ such that $\partial_th $ is supported in $(0,1)\times\T^3$, $h(0,\gamma_0)=\gamma_0,h(1,\gamma_0)=\gamma_1$, and $\bigcup_t h(t,\gamma_0)\subset\T^3\setminus\overline{\Omega_{c}}$.

\item\label{item2.lem.KryGlaHor}
The vector field 
\[
    X(t,x) := \partial_t h(t,h^{-1}(t,x))
\]
is smooth and supported in $(0,1)\times\T^3\setminus\overline{\Omega_{c}}$. Moreover, $X$ satisfies $\nabla\cdot X = 0$ and $\phi^X(1,\gamma_0) = \gamma_1$. Here, $\phi^X$ denotes the flow of $X$.
\end{enumerate}
\end{lemma}

\begin{proof}
The first result (\ref{item1.lem.KryGlaHor}) is due to Krygin \cite{Kry71} and the second (\ref{item2.lem.KryGlaHor}) to Glass-Horsin \cite[Section 2]{GlaHor12}. Although these results are presented for smooth bounded domains $\Omega \subset \R^3$ rather than those in $\T^3$, replacing $\Omega$ with $\T^3\setminus\overline{\Omega_{c}}$ leaves the proof essentially unaffected. 
\end{proof}

The following lemma shares similarity to \cite[Lemma A.1]{Gla03} applied in the study of controllability of the Vlasov–Poisson equations in a periodic domain.

\begin{lemma}\label{lem.Runge.E}
Let $K$ be a closed subset of $\T^3$ such that $K\cap \Omega_c=\emptyset$, and such that $\T^3\setminus K$ is connected. Then, for any harmonic function $\varphi$ in a neighborhood of $K$, $\eta>0$ and $k\in\N$, there exists a smooth $\psi$ on $\T^3$ such that $\opsupp \Delta \psi\subset \overline{\Omega_c}$ and $\|\varphi-\psi\|_{C^k(K)}<\eta$. 
\end{lemma}

\begin{proof}
Fix a point $p\in \Omega_c$ and take an open subset $\Omega_1$ of $\T^3$ such that $p\in \Omega_1$ and $\Omega_1 \Subset \Omega_c$. By the Runge theorem in Bagby-Blanchet \cite[Theorem 9.2]{BagBla94}, there exists a harmonic function $\tilde{\psi}$ defined on $\T^3 \setminus\{p\}$ such that $\|\varphi-\tilde{\psi}\|_{L^{\infty}(K)}<\eta$. Since $\varphi, \tilde{\psi}$ are harmonic in a neighborhood of $K$, according to Schauder estimates, it is possible to improve the topology of the approximation to $\|\varphi-\tilde{\psi}\|_{C^k(K)}<\eta$. Now take a smooth cut-off function $\chi$ on $\T^3$ such that $\chi=1$ on $\T^3\setminus \Omega_c$ and $\chi=0$ on $\Omega_1$. Then $\psi:=\chi \tilde{\psi}$ satisfies the desired properties. 
\end{proof}

We are now in a position to give the proof of Lemma \ref{lem.PotFlows.E}.

\begin{proofx}{Lemma \ref{lem.PotFlows.E}}
The proof is inspired by Glass-Horsin \cite[Proposition 2.2]{GlaHor12}. It suffices to prove the statement under the assumptions that $\gamma_0, \gamma_1$ are real analytic in $\T^3$. 
In addition, we can assume that there exists a real analytic vector field $X\in C^0([0,1]; C^\omega(\T^3)^3)$ such that $\nabla\cdot X = 0$ and the flow $\phi^X$ for $X$, which is a volume preserving real analytic isotopy between $\gamma_0$ and $\gamma_1$, satisfies $\phi^X(1,\gamma_0) = \gamma_1$ and, for some $d_0>0$, 
\begin{equation}\label{est0.prf.lem.PotFlows.E}
    \opdist(\phi^X(t,\gamma_0), \Omega_c)\ge 2d_0, 
    \quad 
    t\in[0,1].
\end{equation}
Indeed, then an approximation procedure as in \cite[Section 2]{GlaHor12} adopting Whitney’s approximation theorem \cite[Proposition 3.3.9]{KraPar2002book} and stability of flows, and Lemma \ref{lem.KryGlaHor} provide the conclusions of Lemma \ref{lem.PotFlows.E}. In the following, we take $\gamma_0, \gamma_1, X, \phi^X$ described as above.

Let $\eta>0$ and $k\in\N$ be given. Set
\[
    \gamma(t) = \phi^X(t,\gamma_0), 
    \quad 
    t\in[0,1]. 
\]
By \cite[Lemmas 2.4]{GlaHor12}, there exist $\xi>0$ and
\[
    \psi\in C
    \big(
    [0, 1]; C^\infty(\mathcal{V}_{\xi}[\opint(\gamma(t))])
    \big), 
\]
where we refer to Appendix \ref{appx.Defs} for the definition of $\opint(\gamma(t))$ and $\mathcal{V}_{\xi}[\opint(\gamma(t))]$ 
denotes the $\xi$-neighborhood of $\opint(\gamma(t))$ in $\T^3$, such that, for each $t\in[0,1]$, 
 \[
    \left\{
    \begin{array}{ll} 
    \Delta \psi = 0 &\mbox{in}\ \mathcal{V}_\xi[\opint(\gamma(t))], \\ [2pt] 
    \displaystyle{\frac{\partial \psi}{\partial \nu}} = X \cdot \nu &\mbox{on}\ \gamma(t). 
    \end{array}\right.
\]
Here $\nu$ is the unit outward normal vector field on $\gamma(t)$. The proof is based on the Cauchy-Kowalevsky theorem and unique continuation. Moreover, we may assume that 
\begin{equation}\label{est1.prf.lem.PotFlows.E}
    \opdist(\mathcal{V}_\xi[\opint(\gamma(t))],\Omega_c)\ge d_0, \quad t\in[0,1]     
\end{equation}
by choosing $\xi$ small enough if necessary.

Then, as in \cite[Proof of Proposition 2.3]{GlaHor12}, by compactness of $[0,1]$, there exist $0\le t_1<\cdots<t_N\le 1$ and $\delta_1,\ldots,\delta_N>0$ such that
\begin{gather}
    [0,1] \subset \bigcup_{i=1}^{N} (t_i - \delta_i,t_i + \delta_i); \label{est2.prf.lem.PotFlows.E} \\
    \gamma(t) \subset \mathcal{V}_{\xi/2}[\gamma(t_i)], 
    \quad 
    t\in[t_i - \delta_i,t_i + \delta_i]; \label{est3.prf.lem.PotFlows.E}\\
    \|\psi(s,\cdot) - \psi(t,\cdot)\|_{C^{k+2}(\overline{\mathcal{V}_{\xi}[\gamma(t_i)]})}<\eta, 
    \quad 
    s,t\in[t_i - \delta_i,t_i + \delta_i]. \label{est4.prf.lem.PotFlows.E}
\end{gather}
For each $i\in \{1,\ldots,N\}$, we set 
\[
    K_i = \overline{\mathcal{V}_{\xi/2}[\opint(\gamma(t_i))]}, 
    \qquad
    \psi_i(x) = \psi(t_i,x). 
\]
Notice that $K_i\cap \Omega_c=\emptyset$ by \eqref{est1.prf.lem.PotFlows.E}. Choosing $\xi$ small enough again if necessary, we may assume that $\T^3\setminus K_i$ is connected for each $i\in \{1,\ldots,N\}$. Then, for any $\tilde{\eta}>0$, applying Lemma \ref{lem.Runge.E} with $K=K_i$ and $\psi=\psi_i$, there exists smooth $\hat{\psi}_i$ defined on $\T^3$ such that
\begin{gather}
    \opsupp \Delta \hat{\psi}_i \subset \overline{\Omega_c}; \label{est5.prf.lem.PotFlows.E}\\
    \|\psi_i-\hat{\psi}_i\|_{C^{k+2}(K_i)}<\tilde{\eta}. \label{est6.prf.lem.PotFlows.E}
\end{gather}
Taking a partition of unity $\{\chi_i\}_{i=1}^{N}$ subordinate to the covering of $[0,1]$ in \eqref{est2.prf.lem.PotFlows.E}, we define 
\[
    \theta(t,x) = \sum_{i=i}^{N} \chi_i(t) \hat{\psi}_i(x). 
\]
Then $\theta$ satisfies \eqref{item1.lem.PotFlows.E} thanks to \eqref{est5.prf.lem.PotFlows.E}. We will show that $\theta$ satisfies \eqref{item2.lem.PotFlows.E}--\eqref{item3.lem.PotFlows.E}.

According to \eqref{est4.prf.lem.PotFlows.E} and \eqref{est6.prf.lem.PotFlows.E}, by choosing $\tilde{\eta}>0$ small depending on $\eta$, we have 
\begin{equation}\label{est7.prf.lem.PotFlows.E}
    \|\nabla\theta(t,\cdot) - \nabla\psi(t,\cdot)\|_{C^{k}(\overline{\mathcal{V}_{\xi/3}[\gamma(t)]})}
    \lesssim \eta, 
    \quad t\in[0,1].
\end{equation}
Let $\phi^{\nabla\theta}$ and $\phi^{\nabla\psi}$ denote the flows of $\nabla\theta$ and $\nabla\psi$, respectively. Note that, provided that $\phi^{\nabla\theta}(t,\gamma_0)$ remains in $\mathcal{V}_{\xi/3}[\gamma(t)]$, Gr\"{o}nwall's inequality leads to 
\begin{equation}\label{est8.prf.lem.PotFlows.E}
    \begin{split}
    &\|\phi^{\nabla\theta}(t,\gamma_0) - \phi^{\nabla\psi}(t,\gamma_0)\|_{C^0} \\
    &\le
    \|\nabla\theta(t,\cdot) - \nabla\psi(t,\cdot)\|_{C^0(\overline{\mathcal{V}_{\xi/3}[\gamma(t)]})}
    \exp \bigg(\int_{0}^{1}\|\nabla\psi(s,\cdot)\|_{C^{1}(\overline{\mathcal{V}_{\xi/3}[\gamma(t)]})}
    \dd s
    \bigg).
    \end{split}
\end{equation}
Using \eqref{est7.prf.lem.PotFlows.E} and replacing $\eta$ by smaller ones if needed, we verify that this inequality holds for all $t\in [0,1]$. Remark that the fact that $\eta$ and $\xi$ are independent is important here. In particular, it is ensured that $\phi^{\nabla\theta}(t,\gamma_0)$ remains in $\mathcal{V}_{\xi/3}[\gamma(t)]$ for all $t\in [0,1]$. Hence \eqref{item2.lem.PotFlows.E} now follows from \eqref{est1.prf.lem.PotFlows.E}. Finally, by differentiating $\phi$ with respect to spatial variables and applying Gr\"{o}nwall's inequality again, we see that, for all $t\in [0,1]$, 
\begin{equation}\label{est9.prf.lem.PotFlows.E}
    \begin{split}
    &\|\phi^{\nabla\theta}(t,\gamma_0) - \phi^{\nabla\psi}(t,\gamma_0)\|_{C^k} \\
    &\le
    \|\nabla\theta(t,\cdot) - \nabla\psi(t,\cdot)\|_{C^k(\overline{\mathcal{V}_{\xi/3}[\gamma(t)]})}
    \exp \bigg(\int_{0}^{1}\|\nabla\psi(s,\cdot)\|_{C^{k+1}(\overline{\mathcal{V}_{\xi/3}[\gamma(t)]})}
    \dd s
    \bigg), 
    \end{split}
\end{equation}
which implies \eqref{item3.lem.PotFlows.E} after replacement of $\eta$. This completes the proof of Lemma \ref{lem.PotFlows.E}. 
\end{proofx}

    \subsection{Homogenization to Euler equations}
    \label{sec.Homogenization.E}

The aim of this section is to prove Proposition \ref{prop.RoC.E} below, which provides the rate of convergence of solutions $u^\eps$ of the Navier-Stokes equations \eqref{eq.NS_fc_E} to the solution $u_E$ of the Euler equations \eqref{eq1.prop.CoE} in Proposition \ref{prop.CoE}. The proof is an adaptation of \cite[Proof of Theorem 1.2]{Hof23} to our setting where the holes of $\Omega^\eps$ are localized in the closed subset $K$.

    \subsubsection{Correctors in partially perforated domains}
    \label{sec.corr}

We introduce the corrector $w^\eps$ for the vector fields on a partially perforated domain $\Omega^\eps$ that have non-zero values on $\partial \Omega^\eps$. The estimates of $w^\eps$ in this section are needed in the proof of Proposition \ref{prop.RoC.E}. Here we follow the presentations mainly in H\"{o}fer \cite[Section 2]{Hof23} where it is assumed that the holes are distributed periodically in $\R^3$ and in Allaire \cite{All90a,All90b}.

Recall from Section \ref{sec.LC.E} that $\mathcal{T}\subset \R^3$ denotes a closed set such that $\mathcal{T}$ is homeomorphic to the three-dimensional closed unit ball, with smooth boundary, containing $0$ in its interior, and contained in $B_{1/8}(0)$. Then we consider the unique solutions $(w_k,q_k)\in \dot{H}^1(\mathbb{R}^3)\times L^2(\mathbb{R}^2)$ to the following steady exterior problem in $\R^3\setminus \mathcal{T}$ for $k=1,2,3$: 
\[
    \left\{
    \begin{array}{ll}
    -\Delta w_k + \nabla q_k = 0 &\mbox{in}\ \R^3\setminus \mathcal{T},\\
    \nabla\cdot w_k = 0 &\mbox{in}\ \R^3\setminus \mathcal{T},\\
    w_k = e_k &\mbox{on}\ \partial\mathcal{T}.
\end{array}\right.
\]
Define the resistance matrix $\mathcal{R} = (\mathcal{R}_{jk})\in\R^{3\times3}$ by 
\begin{equation}\label{resis.mat}
    \mathcal{R}_{jk}
    = \int_{\R^3\setminus \mathcal{T}}
    \nabla w_k \cdot \nabla w_j, 
\end{equation}
which is a positive definite symmetric matrix; see \cite[Proposition 1.1.2]{All90a}.

Let $\alpha>1$. Introducing an intermediate lengthscale 
\begin{equation}\label{lengthscale}
    \eps^\alpha \le \eta^\eps \le \eps, 
\end{equation}
we make the decomposition of $K = \bigcup_i \overline{\mathcal{Q}^\eps_i}$ finer by decomposing $\mathcal{Q}^\eps_i$ as
\[
    \begin{split}
    \mathcal{Q}^\eps_i 
    &= \mathcal{T}^\eps_i 
    \cup \mathcal{C}^\eps_i 
    \cup \mathcal{D}^\eps_i 
    \cup \mathcal{K}^\eps_i, \\
    \mathcal{C}^\eps_i 
    &:= B_{\eta^\eps/4}(x^\eps_i) \setminus \mathcal{T}^\eps_i, \\
    \mathcal{D}^\eps_i 
    &:= B_{\eta^\eps/2}(x^\eps_i) \setminus B_{\eta^\eps/4}(x^\eps_i), \\
    \mathcal{K}^\eps_i 
    &:= \mathcal{Q}^\eps_i \setminus B_{\eta^\eps/2}(x^\eps_i).
    \end{split}
\]
The reader may consult \cite[Figure 2]{Hof23} for visualization.

Define $(w^\eps_k,q_k^\eps)$ for $k=1,2,3$ by 
\[
    \begin{split}
    w^\eps_k = e_k - w_k\Big(\frac{x-x^\eps_i}{\eps^\alpha}\Big),
    \quad
    q_k^\eps = -\eps^\alpha q_k\Big(\frac{x-x^\eps_i}{\eps^\alpha}\Big)
    \quad &\text{in} \mkern9mu \mathcal{C}^\eps_i,\\
    -\Delta w^\eps_k + \nabla q_k^\eps = 0,
    \quad
    \nabla\cdot w^\eps_k = 0
    \quad &\text{in} \mkern9mu \mathcal{D}^\eps_i,\\
    w^\eps_k = e_k,
    \quad
    q_k^\eps = 0
    \quad &\text{in} \mkern9mu \mathcal{K}^\eps_i
    \end{split}
\]
and by $(w^\eps_k,q_k^\eps)=(e_k,0)$ outside $K$. It is emphasized that the boundary condition is imposed on $\partial \mathcal{D}^\eps_i$ which is induced by $w^\eps_k$ on $\mathcal{C}^\eps_i$ and $\mathcal{K}^\eps_i$. Then we define in $\Omega^\eps$
\begin{equation}\label{Corr}
    w^\eps = (w_1^\eps, w_2^\eps, w_3^\eps), 
    \qquad
    q^\eps = (q_1^\eps, q_2^\eps, q_3^\eps).
\end{equation}
One can check that 
$\nabla\cdot w^\eps_k=0$ for $k=1,2,3$, 
$w^\eps\in W^{1,\infty}_0(\Omega^\eps)^{3\times3}$, 
and $q^\eps\in L^{\infty}(\Omega^\eps)^{1\times3}$, 
and moreover, that the following pointwise estimates hold
\[
    \begin{split}
    |{\rm Id} - w^\eps(x)|
    \lesssim
    \frac{\eps^{\alpha}}{|x-x^\eps_i|}
    \quad &\text{in} \mkern9mu \mathcal{C}^\eps_i \cup \mathcal{D}^\eps_i,\\
    |\nabla w^\eps(x)| + |q^\eps(x)|
    \lesssim
    \frac{\eps^{\alpha}}{|x-x^\eps_i|^2}
    \quad &\text{in} \mkern9mu \mathcal{C}^\eps_i \cup \mathcal{D}^\eps_i, 
    \end{split}
\]
which in particular lead to
\[
    \|w^\eps\|_{L^\infty(\Omega^\eps)}
    + \eps^\alpha \big(\|\nabla w^\eps\|_{L^\infty(\Omega^\eps)} + \|q^\eps\|_{L^\infty(\Omega^\eps)}\big)
    \lesssim
    1.
\]
The proof is omitted since it is similar to \cite[Proof of Lemma 2.1 (i)]{Hof23} where holes are distributed periodically in $\R^3$. We extend $(w^\eps, q^\eps)$ to $\T^3$ by zero without changing notation.

In the following, we collect the estimates of $(w^\eps, q^\eps)$ needed in the subsequent sections. Lemma \ref{lem.Corr1} below corresponds to \cite[Lemma 2.2 (ii), (iii)]{Hof23}. The proof is omitted since it is almost identical to the corresponding ones in \cite{Hof23}. Note that $(w^\eps, q^\eps)=({\rm Id},0)$ in $\T^3\setminus K$.

\begin{lemma}\label{lem.Corr1}
The following hold. 
\begin{enumerate}[(1)]
\item\label{item1.lem.Corr1}
For $3/2<p<3$ and $\varphi\in W^{2,p}(\T^3)$,
\begin{equation} \label{ESTI1}
    \|({\rm Id} - w^\eps)\varphi\|_{L^p(\T^3)} 
    \lesssim
    (\eta^\eps)^{\frac{3}{p}-1} \eps^{\alpha-\frac{3}{p}}
    \|\varphi\|_{W^{2,p}(\T^3)}.
\end{equation}

For $\varphi\in W^{2,3}(\T^3)$,
\begin{equation} \label{ESTI2}
    \|({\rm Id} - w^\eps)\varphi\|_{L^3(\T^3)} 
    \lesssim
    \eps^{\alpha-1} |\log \eps|^\frac13
    \|\varphi\|_{W^{2,3}(\T^3)}.
\end{equation}

For $\varphi\in H^{2}(\T^3)$,
\begin{align}
    \||\nabla w^\eps| \varphi\|_{L^2(\T^3)} 
    + \||q^\eps| \varphi\|_{L^2(\T^3)}
    &\lesssim
    \eps^{\frac{\alpha-3}{2}}
    \|\varphi\|_{H^2(\T^3)},\\
    \||\nabla w^\eps|^{\frac12} \varphi\|_{L^2(\T^3)} 
    + \||q^\eps|^{\frac12} \varphi\|_{L^2(\T^3)}
    &\lesssim
    (\eta^\eps)^{\frac12}
    \eps^{\frac{\alpha-3}{2}}
    \|\varphi\|_{H^2(\T^3)}.
\end{align}

\item\label{item2.lem.Corr1}
For $\varphi\in H^1_0(\Omega^{\eps})$,
\begin{equation}
    \||\nabla w^\eps|^{\frac12} \varphi\|_{L^2(\Omega^{\eps})} 
    + \||q^\eps|^{\frac12} \varphi\|_{L^2(\Omega^{\eps})}
    \lesssim
    (\eta^\eps)^{\frac12}
    \|\nabla\varphi\|_{L^2(\Omega^{\eps})}.
\end{equation}
\end{enumerate}
\end{lemma}

Lemma \ref{lem.Corr2} below corresponds to \cite[Lemma 2.2]{Hof23}.

\begin{lemma}\label{lem.Corr2}
There exists a matrix-valued distribution $M^\eps\in W^{-1,\infty}(\Omega^\eps)^{3\times3}$ supported in $K$ such that the following hold.
\begin{enumerate}[(1)]
\item\label{item1.lem.Corr2}
For $V^\eps\in W^{1,\infty}_0(\Omega^\eps)^{3\times3}$, 
\begin{equation}\label{est.item1.lem.Corr2}
    \langle
    -\Delta w^\eps + \nabla q^\eps, V^\eps
    \rangle_{L^2(\Omega^\eps)}
    = 
    \langle
    \eps^{\alpha-3} M^\eps, V^\eps
    \rangle_{L^2(\Omega^\eps)}.
\end{equation}

\item\label{item2.lem.Corr2}
For $\varphi\in H^3(\T^3)^3$ and $\psi\in H^1(\T^3)^3$, 
\begin{equation}\label{est.item2.lem.Corr2}
    \begin{split}
    &|\langle
    (M^\eps - \mathcal{R} \mathbbm{1}_{K})\varphi, \psi
    \rangle_{L^2(\Omega^\eps)}| \\
    &\lesssim 
    \big(
    (\eta^\eps)^{-1} \eps^\alpha \|\psi\|_{L^2(\T^3)}
    + (\eta^\eps)^{-\frac12} \eps^{\frac32} \|\psi\|_{H^1(\T^3)}
    \big)
    \|\varphi\|_{H^3(\T^3)},
    \end{split}
\end{equation}
where $\mathcal{R}$ denotes the resistance matrix in \eqref{resis.mat}.
\end{enumerate}
\end{lemma}

\begin{proof}
We only give the outline since it is almost parallel to \cite[Lemma 2.2]{Hof23}. Thus we only highlight the parts that are concerned with the localization of holes in $K$.

(\ref{item1.lem.Corr2}) 
By definition, $-\Delta w^{\eps}+\nabla  q^{\eps}$ is supported on $\bigcup_{i}(\partial\mathcal{C}^{\eps}_i\cup\partial\mathcal{D}^{\eps}_i)=\bigcup_i\partial \mathcal{D}_i^{\eps}\cup \partial\Omega^{\eps}$. We define $\eps^{\alpha-3} M^{\eps}$ to be the part of it that is supported on $\bigcup_i\partial \mathcal{D}^{\eps}_i$. 
Then \eqref{est.item1.lem.Corr2} holds true.

(\ref{item2.lem.Corr2}) 
Recall from \cite[Subsection 2.3.2]{All90a} that $M^\eps = (M^\eps_1, M^\eps_2, M^\eps_3)$ is expressed as 
\begin{equation*}\label{est1.prf.lem.Corr2}
    M^\eps_k
    = \eps^{3-\alpha} \sum_{i}
    \Big(
    m^\eps_{k,i} + \nabla\cdot\big(\mathbbm{1}_{\mathcal{D}^\eps_i}(-\nabla w^\eps_k+q^\eps_k {\rm Id})\big)
    \Big),
    \quad
    k=1,2,3. 
\end{equation*}
The vector-valued distribution $m^\eps_{k,i}$ is given by $m^\eps_{k,i} = m^{\eps,1}_{k,i} + m^{\eps,2}_{k,i}$ and 
\begin{equation*}\label{est2.prf.lem.Corr2}
    \begin{split}
    m^{\eps,1}_{k,i}
    &:= \frac{\eps^\alpha}{2}
    \big(
    \mathcal{R}_{k} + 3(\mathcal{R}_{k} \cdot n) n 
    \big)
    \delta^i_{\eta^\eps/4}, 
   \qquad
    m^{\eps,2}_{k,i}
    := \frac{(\eta^\eps)^{-1} \eps^{2\alpha}}{2}
    r^\eps_{k,i}
    \delta^i_{\eta^\eps/4}. 
    \end{split}
\end{equation*}
Here $n$, $\delta^i_{\eta^\eps/4}$ and $r^\eps_{k,i}$ denote the outward unit normal vector to $\partial B_{\eta^\eps/4}(x^\eps_i)$, the unit mass measure concentrated on $\partial B_{\eta^\eps/4}(x^\eps_i)$, and a term satisfying $\|r^\eps_{k,i}\|_{W^{1,\infty}(\partial B_{\eta^\eps/4}(x^\eps_i))} \lesssim 1$, respectively. Then \eqref{est.item2.lem.Corr2} follows from the following three estimates: 
\begin{equation*}
    \begin{split}
    &\bigg|
    \bigg\langle
    \bigg(
    \eps^{3-\alpha} \sum_{i} m^{\eps,1}_{k,i} 
    - \mathcal{R}_k \mathbbm{1}_{K}
    \bigg) \varphi, 
    \psi
    \bigg\rangle_{L^2(\mathbb{T}^3)}
    \bigg| \\
    &\quad\lesssim
    (\eta^\eps)^{-\frac12} \eps^{\frac32}
    \|\psi\|_{H^1(\T^3)}
    \|\varphi\|_{H^3(\T^3)}, \\
    &\bigg|
    \bigg\langle
    \bigg(
    \eps^{3-\alpha} \sum_{i} m^{\eps,2}_{k,i} 
    \bigg) \varphi, 
    \psi
    \bigg\rangle_{L^2(\mathbb{T}^3)}
    \bigg| \\
    &\quad\lesssim
    (\eta^\eps)^{-1} \eps^\alpha 
    \big(
    \|\psi\|_{L^2(\T^3)}
    + (\eta^\eps)^{-\frac12} \eps^{\frac32} \|\psi\|_{H^1(\T^3)}
    \big)
    \|\varphi\|_{H^3(\T^3)}, \\
    &\bigg|
    \bigg\langle
    \bigg(
    \eps^{3-\alpha} \sum_{i} \nabla\cdot\big(\mathbbm{1}_{\mathcal{D}^\eps_i}(-\nabla w^\eps_k+q^\eps_k {\rm Id})\big) 
    \bigg) \varphi, 
    \psi
    \bigg\rangle_{L^2(\mathbb{T}^3)}
    \bigg| \\
    &\quad\lesssim
    (\eta^\eps)^{-\frac12} \eps^{\frac32}
    \|\psi\|_{H^1(\T^3)}
    \|\varphi\|_{H^3(\T^3)}, 
    \end{split}
\end{equation*} 
combined with $(\eta^\eps)^{-1} \eps^\alpha \le 1$ due to the definition of $\eta^\eps$ in \eqref{lengthscale}. The reader is referred to \cite[Proof of Lemma 2.2]{Hof23} for the necessary details. We conclude the proof.
\end{proof}

Lemma \ref{lem.Corr3} below corresponds to \cite[Lemma 2.3]{Hof23}. The proof is omitted since it is almost identical to \cite[Proof of Lemma 2.3]{Hof23}.

\begin{lemma}\label{lem.Corr3}
Let $1<p<+\infty$. There exists a linear operator $\opbogeps$ mapping $\varphi\in W^{1,p}(\T^3)^3$ with $\nabla\cdot \varphi = 0$ to an element of $\opbogeps(\varphi) \in W^{1,p}(\T^3)^3$ such that the following hold: 
\begin{equation}\label{item1.lem.Corr3}
    \nabla\cdot \opbogeps(\varphi) = w^\eps \cdot \nabla \varphi
\end{equation}
and
\begin{equation}\label{item2.lem.Corr3}
    \|\nabla^j \opbogeps(\varphi)\|_{L^p(\T^3)}
    \lesssim (\eta^\eps)^{1-j} \|({\rm Id} - w^\eps) \cdot \nabla \varphi\|_{L^p(\T^3)},
    \quad
    j=0,1.
\end{equation}
\end{lemma}

    \subsubsection{Rate of convergence}

Let $(u_E, p_E, f_E)$ be the triplet satisfying the Euler equations in Proposition \ref{prop.CoE}. Set 
\[
    f^{\eps} = f_E|_{\Omega^\eps}. 
\]
Then we consider the Navier-Stokes equations in $\Omega^{\eps}$ with viscosity $\eps^\beta$ and forcing $f^{\eps}$ 
\begin{equation}\label{eq.NS_fc_E}
    \left\{
    \begin{array}{ll}
    \partial_t u^\eps - \eps^\beta \Delta u^\eps + u^\eps\cdot\nabla u^\eps + \nabla p^\eps 
    = f^{\eps} &\mbox{in}\ (0,1)\times\Omega^\eps,\\
    \nabla\cdot u^\eps = 0 &\mbox{in}\ (0,1)\times\Omega^\eps,\\
    u^\eps = 0 &\mbox{on}\ (0,1)\times\partial\Omega^\eps,\\
    u^\eps = u^\eps_0 &\mbox{on}\ \{0\}\times\Omega^\eps.
    \end{array}\right.
\end{equation}
For this section, we will work with general initial data $u^\eps_0 \in L^2_\sigma(\Omega^\eps)$ and Leray-Hopf weak solutions of \eqref{eq.NS_fc_E} that satisfy the energy inequality 
\begin{equation}\label{EnIneq}
    \begin{split}
    &\frac12 \|u^\eps(t)\|_{L^2(\Omega^\eps)}^2 
    + \eps^\beta 
    \int_0^t 
    \|\nabla u^\eps(s)\|_{L^2(\Omega^\eps)}^2 
    \dd s \\
    &\le
    \frac12 \|u^\eps_0\|_{L^2(\Omega^\eps)}^2 
    + \int_0^t \int_{\Omega^\eps}
    f^{\eps} \cdot u^\eps
    \dd x \dd s,
    \quad
    t\in [0,1].
    \end{split}
\end{equation}

The following proposition shows that Leray-Hopf weak solutions $u^\eps$ of \eqref{eq.NS_fc_E} converge to $u_E$ in the $L^2$-sense under smallness in $\eps$ on the norm of initial data $\|u^\eps_0\|_{L^2(\Omega^\eps)}$.

\begin{proposition}\label{prop.RoC.E}
For $0<\eps<1$, we have 
\begin{equation}\label{est1.prop.RoC.E}
    \begin{split}
    \|u^\eps(t)-u_E(t)\|_{L^2(\Omega^\eps)}
    &\lesssim
    \|u^\eps_0\|_{L^2(\Omega^\eps)} 
    + \eps^{\mathfrak p_E},
    \quad
    t\in [0,1], 
    \end{split}
\end{equation}
where the exponent $\mathfrak p_E$ is defined in \eqref{def.pE}.
\end{proposition}

\begin{proof}
We only give the outline since it is almost parallel to \cite[Proof of Theorem 1.2]{Hof23}. Thus estimates such as embeddings used in \cite{Hof23} will not be repeated here. Also, we focus solely on the case when $\eps$ is sufficiently small; otherwise, the statement is trivial.

Define
\[
     \Check{u}^\eps = w^\eps u_E - \opbogeps(u_E), 
     \qquad
     v^\eps = \Check{u}^\eps - u^\eps,
\]
where $w^\eps$ is the corrector in \eqref{Corr} and $\opbogeps(\,\cdot\,)$ the operator in Lemma \ref{lem.Corr3}. From 
\[
    u^\eps - u_E = -({\rm Id} - w^\eps)u_E - \opbogeps(u_E) - v^\eps,
\]
we see that, by using Lemmas \ref{lem.Corr1} and \ref{lem.Corr3}, 
\begin{equation}\label{est0.prf.prop.RoC.E}
    \begin{split}
    \|u^\eps(t)-u_E(t)\|_{L^2(\Omega^\eps)}
    \lesssim
    (\eta^\eps)^{\frac12} \eps^{\alpha-\frac32} 
    + \|v^\eps(t)\|_{L^2(\Omega^\eps)}
    \lesssim
    \eps^{\alpha-1} 
    + \|v^\eps(t)\|_{L^2(\Omega^\eps)}. 
    \end{split}
\end{equation}
Recall that $\eta^\eps\le\eps$ is the intermediate lengthscale introduced in \eqref{lengthscale}.

Thus we focus on estimating $\|v^\eps(t)\|_{L^2(\Omega^\eps)}$. We first observe that $\Check{u}^\eps$ satisfies
\[
    \left\{
    \begin{array}{ll}
    \partial_t \Check{u}^\eps - \eps^\beta \Delta \Check{u}^\eps + w^\eps (u_E\cdot\nabla u_E) + \nabla p_E
    = w^\eps f^\eps + g^\eps
    &\mbox{in}\ (0,1)\times\Omega^\eps,\\
    \nabla\cdot \Check{u}^\eps = 0 &\mbox{in}\ (0,1)\times\Omega^\eps,\\
    \Check{u}^\eps = 0 &\mbox{on}\ (0,1)\times\partial\Omega^\eps,\\
    \Check{u}^\eps = 0 &\mbox{on}\ \{0\}\times\Omega^\eps
    \end{array}\right.
\]
with the term $g^\eps$ involving $q^\eps$ in \eqref{Corr} defined by
\[
    \begin{split}
    g^\eps
    &= 
    ({\rm Id} - w^\eps) \nabla p_E 
    - \eps^\beta  \Delta (w^\eps u_E) 
    - (\partial_t - \eps^\beta \Delta) \opbogeps(u_E) \\
    &= 
    ({\rm Id} - w^\eps) \nabla p_E
    - \eps^\beta 
    (\Delta w^\eps - \nabla q^\eps) u_E \\
    &\quad
    - \eps^\beta 
    \big(
    2\nabla w^\eps\cdot\nabla u_E + w^\eps \Delta u_E + (\nabla q^\eps) u_E 
    \big) 
    - (\partial_t - \eps^\beta \Delta) \opbogeps(u_E). 
    \end{split}
\]
Here the term $\Delta w^\eps - \nabla q^\eps$ should be understood in the sense of distributions. Then, by direct computation and $\Check{u}^\eps(0)=0$, we see that $v^\eps$ satisfies
\[
    \left\{
    \begin{array}{ll}
    \partial_t v^\eps - \eps^\beta \Delta v^\eps + \nabla (p_E - p^\eps)
    = G^\eps
    &\mbox{in}\ (0,1)\times\Omega^\eps,\\
    \nabla\cdot v^\eps = 0 &\mbox{in}\ (0,1)\times\Omega^\eps,\\
    v^\eps = 0 &\mbox{on}\ (0,1)\times\partial\Omega^\eps,\\
    v^\eps = -u^\eps_0 &\mbox{on}\ \{0\}\times\Omega^\eps 
    \end{array}\right.
\]
with
\[
    G^\eps 
    = (\Check{u}^\eps - v^\eps)\cdot\nabla (\Check{u}^\eps - v^\eps)
    - w^\eps (u_E\cdot\nabla u_E) 
    - ({\rm Id} - w^\eps) f^\eps + g^\eps. 
\]
Hence we obtain the inequality
\begin{equation}\label{est6.prf.prop.RoC.E}
    \begin{split}
    \frac12 \|v^\eps(t)\|_{L^2(\Omega^\eps)}^2
    + \eps^\beta 
    \int_0^t 
    \|\nabla v^\eps(s)\|_{L^2(\Omega^\eps)}^2
    \dd s 
    \le
    \frac12 \|u^\eps_0\|_{L^2(\Omega^\eps)}^2
    + |I_1(t)| + |I_2(t)|, 
    \end{split}
\end{equation}
where
\begin{equation*}
\begin{split}
    I_1(t)
    &:= 
    \int_0^t \int_{\Omega^\eps}
    \big(
    -v^\eps\cdot\nabla \Check{u}^\eps
    + \Check{u}^\eps\cdot\nabla \Check{u}^\eps
    - w^\eps(u_E\cdot\nabla u_E)
    \big)
    \cdot v^\eps
    \dd x \dd s, \\
    I_2(t)
    &:= 
    \int_0^t \int_{\Omega^\eps}
    \big(
    -({\rm Id} - w^\eps) f^\eps + g^\eps 
    \big)
    \cdot v^\eps
    \dd x \dd s.
    \end{split}
\end{equation*}
Notice that some cancellations in nonlinearity are implicitly used in $I_1(t)$.

Then, by smoothness of $u_E, p_E$ and the estimates for $w^\eps, q^\eps$ in Section \ref{sec.corr}, we verify  
\begin{equation}\label{est7.prf.prop.RoC.E}
    \begin{split}
    &|I_1(t)| + |I_2(t)| \\
    &\le
    C \int_0^t 
    \|v^\eps(s)\|_{L^2(\Omega^\eps)}^2 \dd s
    + \Big(\frac{\eps^\beta}{4} + C\eta^\eps\Big) 
    \int_0^t 
    \|\nabla v^\eps(s)\|_{L^2(\Omega^\eps)}^2
    \dd s \\
    &\quad
    + C 
    \big(
    \eta^\eps \eps^{2\alpha-\beta-3}
    + (\eta^\eps)^{-1} \eps^{2\alpha+\beta-3}
    + \eps^{2\alpha+2\beta-6}
    + \eps^{2\beta}
    + (\eta^\eps)^2
    \big)
    \end{split}
\end{equation}
for some $C>0$ independent of $\eps$. The details are omitted since they are similar to those in \cite[Proof of Proposition 3.1 (ii)]{Hof23} where the holes are distributed periodically in $\R^3$.

Now we choose the intermediate lengthscale $\eps^\alpha \le \eta^\eps \le \eps$ in \eqref{lengthscale} so that 
\[
    \eta^\eps \simeq o(\eps^{\max\{\beta,1\}})
\]
to ensure that $\eps^\beta/4 + C\eta^\eps \le \eps^\beta/2$ and that 
\[
    \eta^\eps \eps^{2\alpha-\beta-3}
    + (\eta^\eps)^{-1} \eps^{2\alpha+\beta-3}
    + (\eta^\eps)^2
    \lesssim
    \eps^{2\alpha+2\beta-6} 
    + \eps^{2\alpha-3}
    + \eps^{2\beta}.
\]
Then the estimates \eqref{est6.prf.prop.RoC.E}--\eqref{est7.prf.prop.RoC.E} lead to 
\begin{equation*}
    \begin{split}
    &\|v^\eps(t)\|_{L^2(\Omega^\eps)}^2
    + \eps^\beta 
    \int_0^t 
    \|\nabla v^\eps(s)\|_{L^2(\Omega^\eps)}^2
    \dd s \\
    &\lesssim
    \int_0^t 
    \|v^\eps(s)\|_{L^2(\Omega^\eps)}^2 \dd s
    + \|u^\eps_0\|_{L^2(\Omega^\eps)}^2
    + \eps^{2\alpha+2\beta-6}
    + \eps^{2\alpha-3}
    + \eps^{2\beta}.
    \end{split}
\end{equation*}
Moreover, Gr\"{o}nwall's inequality and taking square root give
\begin{equation}\label{est8.prf.prop.RoC.E}
    \begin{split}
    \|v^\eps(t)\|_{L^2(\Omega^\eps)}
    \lesssim
    \|u^\eps_0\|_{L^2(\Omega^\eps)}
    + \eps^{\alpha+\beta-3}
    + \eps^{\alpha-\frac32}
    + \eps^{\beta}.
    \end{split}
\end{equation}
Hence the assertion follows from the inequalities \eqref{est0.prf.prop.RoC.E} and \eqref{est8.prf.prop.RoC.E}.
\end{proof}

    \subsection{Proof of Theorem \ref{thm.main.E}}
    \label{sec.Scaling-Stability.E}

In this section, we prove Theorem \ref{thm.main.E}. Define $(U^\eps, P^\eps, F^{\eps})$ by the scaling \eqref{eq.UPF_E}. Then $(U^\eps, P^\eps, F^{\eps})$ satisfies the Navier-Stokes equations in $\Omega^{\eps}$ with viscosity $1$ and forcing $F^{\eps}$ 
\begin{equation}
    \left\{
    \begin{array}{ll}
    \partial_t U^\eps - \Delta U^\eps + U^\eps\cdot\nabla U^\eps + \nabla P^\eps 
    = F^{\eps} &\mbox{in}\ (0,\eps^{\beta})\times\Omega^\eps,\\
    \nabla\cdot U^\eps = 0 &\mbox{in}\ (0,\eps^{\beta})\times\Omega^\eps,\\
    U^\eps = 0 &\mbox{on}\ (0,\eps^{\beta})\times\partial\Omega^\eps,\\
    U^\eps = U^\eps_0 &\mbox{on}\ \{0\}\times\Omega^\eps.
    \end{array}\right.
\end{equation}

Let $\phi^\eps, \Phi^\eps$ denote the flows on $\Omega^\eps$ for $u^\eps, U^\eps$, respectively, and let $\phi_E$ denote the flow on $\T^3$ for $u_E$. Since $u^\eps(t,x)=\eps^\beta U^\eps(\eps^\beta t,x)$ by \eqref{eq.UPF_E}, we have 
\begin{equation}\label{eq.FlowMaps_E}
    \phi^\eps(t,x) = \Phi^\eps(\eps^\beta t,x),
    \quad
    t\in [0,1]. 
\end{equation}

We first prove the following lemmas for $\phi^\eps, \phi_E$. Recall that $\mathfrak p_E$ is defined in \eqref{def.pE}.

\begin{lemma}\label{lem.RoC.phiE}
For  $0<\eps<1$, we have
\begin{equation}\label{est1.lem.RoC.phiE}
\begin{split}
    \|\phi^\eps(t) - \phi_E(t)\|_{L^2(\Omega^\eps)}
    \lesssim
    \eps^\beta \|U^\eps_0\|_{L^2(\Omega^\eps)} + \eps^{\mathfrak p_E},
    \quad
    t\in [0,1]. 
\end{split}
\end{equation}
\end{lemma}

\begin{proof}
We compute 
\begin{equation*}
\begin{split}
    \phi^\eps(t,x) - \phi_E(t,x)
    &= 
    \int_0^t 
    u^\eps(s,\phi^\eps(s,x))
    \dd s
    - \int_0^t 
    u_E(s,\phi_E(s,x))
    \dd s \\
    &= 
    \int_0^t 
    \big(
    u^\eps(s,\phi^\eps(s,x))
    - u_E(s,\phi^\eps(s,x))
    \big)
    \dd s \\
    &\quad 
    + \int_0^t 
    \big(
    u_E(s,\phi^\eps(s,x))
    - u_E(s,\phi_E(s,x))
    \big)
    \dd s.
\end{split}
\end{equation*}
Smoothness of $u_E$ leads to 
\begin{equation*}
\begin{split}
    |u_E(t,\phi^\eps(t,x))- u_E(t,\phi_E(t,x))|
    &\le
    \|\nabla u_E(t)\|_{L^\infty(\T^3)}
    |\phi^\eps(t,x) - \phi_E(t,x)|, 
\end{split}
\end{equation*}
which results in
\begin{equation*}
\begin{split}
    |\phi^\eps(t,x) - \phi_E(t,x)|
    &\lesssim
    \int_0^t
    |u^\eps(s,\phi^\eps(s,x)) - u_E(s,\phi^\eps(s,x))|
    \dd s \\
    &\quad
    + \int_0^t
    |\phi^\eps(s,x) - \phi_E(s,x)|
    \dd s.
\end{split}
\end{equation*}
By Gr\"{o}nwall's inequality, there is $C>0$ such that, for almost every $x\in \mathbb{T}^3$, we have
\begin{equation}\label{est1.prf.lem.RoC.phiE}
    \begin{split}
    &|\phi^\eps(t,x) - \phi_E(t,x)| \\
    &\le
    C\int_0^t e^{C(t-s)} 
    |u^\eps(s,\phi^\eps(s,x)) - u_E(s,\phi^\eps(s,x))|
    \dd s. 
    \end{split}
\end{equation}
Take the $L^2$-norm in space in \eqref{est1.prf.lem.RoC.phiE} and use the fact that
\begin{equation*}
    \begin{split}
    \|
    u^\eps(t,\phi^\eps(t,\,\cdot\,)) 
    - u_E(t,\phi^\eps(t,\,\cdot\,))
    \|_{L^2(\Omega^\eps)}
    = 
    \|u^\eps(t) - u_E(t)\|_{L^2(\Omega^\eps)},
    \end{split}
\end{equation*}
which follows from incompressibility of the velocity fields and the boundary condition of $u^{\eps}$ entailing the consistency of integral region. Then we see that, for $t\in[0,1]$, 
\[
    \|\phi^\eps(t) - \phi_E(t)\|_{L^2(\Omega^\eps)} \le
    C\int_0^t 
    \|u^\eps(s) - u_E(s)\|_{L^2(\Omega^\eps)}
    \dd s.
\]
Hence, with 
\[
     u^\eps_0=\eps^\beta U^\eps_0, 
\]
the assertion follows from Proposition \ref{prop.RoC.E}. 
\end{proof}

Recall that $0<L<1$ denotes the length of the set $K$ where the holes are distributed.

\begin{lemma}\label{lem.Markov.E}
For given $0<\eta<1$, take $\eps_0$ defined in \eqref{def.eps_E}. 
Then, for any $0 < \eps \le \eps_0$ of the form \eqref{def.eps} and $U^\eps_0$ satisfying $\|U^\eps_0\|_{L^2(\Omega^\eps)} \le \eps^{\mathfrak p_E - \beta}$, we have
\[
    \begin{split}
    \mathcal{L}\big(
    \{x\in \Omega^\eps~|~|\phi^\eps(t,x) - \phi_E(t,x)| > \eta\}
    \big)
    \lesssim 
    \eta,
    \quad
    t\in [0,1].
    \end{split}
\]
\end{lemma}

\begin{proof}
By Markov's inequality, H\"{o}lder's inequality and Lemma \ref{lem.RoC.phiE}, we have
\begin{equation*}
    \begin{split}
    &\mathcal{L}\big(\{x\in \Omega^\eps~|~|\phi^\eps(t,x) - \phi_E(t,x)|>\eta\}\big) \\
    &\le
    \eta^{-1} 
    \|\phi^\eps(t) - \phi_E(t)\|_{L^1(\Omega^\eps)} \\
    &\lesssim
    \eta^{-1} 
    \big(
    \eps^\beta \|U^\eps_0\|_{L^2(\Omega^\eps)} + \eps^{\mathfrak p_E}
    \big)
    \lesssim
    \eta^{-1} \eps^{\mathfrak p_E}.
    \end{split}
\end{equation*}
The assertion follows from the definition of $\eps_0$. 
\end{proof}

Now we prove Theorem \ref{thm.main.E}.

\begin{proofx}{Theorem \ref{thm.main.E}}
Let $\eta>0$ be given. We aim at proving that $\mathcal{L}(P_1 \setminus \Phi^\eps(\eps^{\beta},P_0^\eps)) \lesssim \eta$. By the relation \eqref{eq.FlowMaps_E}, it suffices to prove that $\mathcal{L}(P_1 \setminus \phi^\eps(1,P_0^\eps)) \lesssim \eta$.

Observe that 
\begin{equation}\label{est1.prf.thm.main.E}
\begin{split}
    \mathcal{L}\big(P_1\setminus \phi^\eps(1,P_0^{\eps})\big)
    &\le
    \mathcal{L}\big(P_1\setminus \phi_E(1,P_0)\big)
    + \mathcal{L}\big(\phi_E(1,P_0)\setminus\phi_E(1, P_0^{\eps})\big)\\
   &\quad + \mathcal{L}\big(\phi_E(1,P_0^\eps)\setminus \phi^\eps(1,P_0^\eps)\big). 
   \end{split}
\end{equation}
The properties of $\phi_E$ in Proposition \ref{prop.CoE} show that 
\begin{equation}\label{est2.prf.thm.main.E}
    \mathcal{L}(P_1\setminus \phi_E(1,P_0))
    \le
    \eta. 
\end{equation}
Thanks to $\nabla\cdot u_E=0$, the volume of a set $A$ is preserved by $\phi_E(t,A)$: 
\begin{equation}\label{est3.prf.thm.main.E}
    \begin{split}
    \mathcal{L}\big(\phi_E(1,P_0)\setminus\phi_E(1, P_0^{\eps})\big)
    &=\mathcal{L}( P_0\setminus P_0^{\eps})\\
    &\le \mathcal{L}\Big(\bigcup_iT^{\eps}_i\Big)
    \lesssim \Big(\frac{L}{\eps}\Big)^3 \eps^{3\alpha}
    \lesssim \eps^{\frac{\mathfrak{p}_E}{2}}
    \lesssim \eta.
    \end{split}
\end{equation}

To estimate the last term in the right of \eqref{est1.prf.thm.main.E}, we use
\[
    \mathcal{L}\big(\phi_E(1,P_0^\eps)\setminus \phi^\eps(1,P_0^\eps)\big) 
    =\mathcal{L}\big(\phi^\eps(1,P_0^\eps)\setminus \phi_E(1,P_0^\eps)\big),
\]
which follows from the combination of
\[
    \mathcal{L}(P_0^\eps) 
    = \mathcal{L}\big(\phi^\eps(t,P_0^\eps)\big) 
    = \mathcal{L}\big(\phi_E(t,P_0^\eps)\big), 
\]
valid thanks to $\nabla\cdot u^\eps=\nabla\cdot u_E=0$, and 
\[
    \begin{split}
    \mathcal{L}\big(\phi^\eps(t,P_0^\eps)\big) 
    &=\mathcal{L}\big(\phi^\eps(t,P_0^\eps)\setminus \phi_E(t,P_0^\eps)\big)
    + \mathcal{L}\big(\phi^\eps(t,P_0^\eps)\cap \phi_E(t,P_0^\eps)\big), \\
    \mathcal{L}\big(\phi_E(t,P_0^\eps)\big) 
    &=\mathcal{L}\big(\phi_E(t,P_0^\eps)\setminus \phi^\eps(t,P_0^\eps)\big)
    + \mathcal{L}\big(\phi^\eps(t,P_0^\eps)\cap \phi_E(t,P_0^\eps)\big).
    \end{split}
\]
Then we have
\[
    \begin{split}
    &\mathcal{L}\big(\phi_E(1,P_0^\eps)\setminus \phi^\eps(1,P_0^\eps)\big) 
    =\mathcal{L}\big(\phi^{\eps}(1,P_0^\eps)\setminus \phi_E(1,P_0^\eps)\big)
    \\
    &\le
    \mathcal{L}\big(\phi^\eps(1,P_0^\eps)\setminus
    \mathcal{V}_\eta[\phi_E(1,P_0^\eps)]\big)
    + \mathcal{L}\big(\mathcal{V}_\eta[\phi_E(1,P_0^\eps)]\setminus \phi_E(1,P_0^\eps)\big).
    \end{split}
\]
Using the inclusion
\[
    \phi^\eps(t,P_0^\eps) \setminus \mathcal{V}_{\eta}[\phi_E(t,P_0^\eps)]
    \subset
    \{x\in P_0^\eps~|~|\phi^\eps(t,x)-\phi_E(t,x)|\ge \eta\},
\]
an estimate valid thanks to smoothness of $u_E$
\[
    \mathcal{L}\big(\mathcal{V}_\eta[\phi_E(1,P_0^\eps)]\setminus \phi_E(1,P_0^\eps)\big)
    \lesssim
    \eta, 
\]
and Lemma \ref{lem.Markov.E}, we obtain
\begin{equation}\label{est4.prf.thm.main.E}
    \begin{split}
    &\mathcal{L}\big(\phi_E(1,P_0^\eps)\setminus \phi^\eps(1,P_0^\eps)\big) \\
    &\lesssim
    \mathcal{L}\big(\{x\in P_0~|~|\phi^\eps(1,x) - \phi_E(1,x)|\ge \eta\}\big) + \eta 
    \lesssim
    \eta.
    \end{split}
\end{equation}
The assertion of Theorem \ref{thm.main.E} follows from \eqref{est1.prf.thm.main.E}--\eqref{est4.prf.thm.main.E}. The proof is complete.
\end{proofx}

    \section{Lagrangian controllability in fully perforated domains}
    \label{sec.LC.D}

Let us consider the 3d Navier-Stokes equations with viscosity $1$ and forcing $F^\eps$, namely \eqref{eq.NS}. Assuming $\Omega^\eps$ to be a domain fully perforated, we are interested in the following Lagrangian controllability problem: move the fluid particles which are initially in a zone $P_0^\eps$ toward a target zone $P_1^\eps$ in some time $T_c>0$. The way to move the particles is by remote action: regard the term $F^\eps$ in \eqref{eq.NS} as a control forcing moving $P_0^\eps$ to $P_1^\eps$ and assume that $F^\eps$ is smooth and supported in a control zone $\Omega_c^\eps$. The control zone $\Omega_c^\eps$ is supposed to be away both from $P_0^{\eps}$ and $P_1^\eps$. The precise definitions $\Omega^\eps, P_0^\eps, P_1^\eps$, and $\Omega_c^\eps$ are as follows.

\begin{itemize}
\item
\textbf{Definition of a fully perforated domain $\Omega^\eps$} \\
Take $1 < \alpha < 3$ and sufficiently large $N\in\N$, and decompose $\T^3$ as  
$\T^3 = \bigcup_i \overline{\mathcal{Q}^\eps_i}$ with $N^3$ small open disjoint small cubes $\mathcal{Q}^\eps_i$ with center $x^\eps_i$ and sidelength 
\begin{equation}\label{def.eps'}
    \eps = \frac{1}{N}. 
\end{equation}
Choose the same reference particle $\mathcal{T}$ as in Section \ref{sec.LC.E}. Then we consider $N^3$ particles $\mathcal{T}^\eps_i = x^\eps_i + \eps^\alpha \mathcal{T} \subset \mathcal{Q}^\eps_i$ and let 
\[
    \Omega^\eps = \T^3 \setminus \bigcup_i \mathcal{T}^\eps_i. 
\]
The set $\Omega^\eps$ obtained in this manner is called a fully perforated domain.

\item
\textbf{Definition of a polluted zone $P_0^\eps$ and a safe zone $P_1^\eps$} \\
Let each of $P_0$ and $P_1$ be a Jordan domain in $\T^3$. Let $\gamma_0$ and $\gamma_1$ denote the Jordan surfaces of $P_0$ and $P_1$, respectively. Assume that $\gamma_0$ and $\gamma_1$ are isotopic in $\T^3$ and surround the same volume. Then we set $P_0^\eps = P_0\cap \Omega^\eps$ and $P_1^\eps = P_1\cap \Omega^\eps$.

\item
\textbf{Definition of a control zone $\Omega_c^\eps$} \\
Let $\Omega_c\neq \emptyset$ be an open set in $\T^3$ such that $\overline{\Omega_{c}} \cap P_0 = \overline{\Omega_{c}} \cap P_1 = \emptyset$.  Notice that $\gamma_0$ and $\gamma_1$ above are then isotopic in $\T^3\setminus\overline{\Omega_{c}}$. By the same reasoning as in Section \ref{sec.LC.E}, we assume that $\Omega_c$ is a ball with smooth boundary. Then we set $\Omega_c^\eps = \Omega_c\cap \Omega^\eps$. 
\end{itemize}

Now let us state two results of the Lagrangian controllability for Leray-Hopf weak solutions of the Navier-Stokes equations \eqref{eq.NS} in a fully perforated domain $\Omega^\eps$.

\begin{theorem}\label{thm.main.E'}
Let $\alpha>3/2$ and choose $\beta>0$ so that $3-\alpha < \beta < \alpha$, which ensures that $\mathfrak p_E$ in \eqref{def.pE} is positive. For given $0<\eta<1$, we define
\begin{equation}\label{def.eps_E'}
    \eps_0=\frac{1}{\lfloor \eta^{-2/\mathfrak{p}_E}\rfloor+1}, 
\end{equation}
where $\lfloor\,\cdot\,\rfloor$ refers to the floor function. 
Then, for any $0 < \eps \le \eps_0$ of the form \eqref{def.eps'} and initial data $U^\eps_0\in L^2_\sigma(\Omega^\eps)$ satisfying
\begin{equation}\label{est1.thm.main.E'}
    \|U^\eps_0\|_{L^2(\Omega^\eps)}
    \le
    \eps^{\mathfrak p_E - \beta}, 
\end{equation}
the Lagrangian controllability of \eqref{eq.NS} in the time $\eps^\beta$ holds: there exists a smooth control forcing $F^\eps$ supported in $\Omega_c^\eps$ with $\|F^\eps\|_{L^{\infty}_t H^s_x} \le C_0 \eps^{-2\beta}$ for any $s\in\N$ 
and a constant $C_0=C_0(s)>0$ depending on $s$
such that 
\begin{equation}\label{est2.thm.main.E'}
    \bigcup_{t\in [0,\eps^{\beta}]} \Phi^\eps(t,P_0^\eps) \subset\Omega^\eps\setminus\overline{\Omega_{c}^\eps}, \qquad 
    \mathcal{L}(P_1^\eps \setminus \Phi^\eps(\eps^\beta,P_0^\eps)) 
    \le C \eta, 
\end{equation}
where $\mathcal{L}$ refers to the Lebesgue measure and $\Phi^\eps$ the flow for a corresponding Leray-Hopf weak solution $U^\eps$ to \eqref{eq.NS}. The constant $C>0$ is independent of $\eta$ and $\eps$. 
\end{theorem}

\begin{proof}
Theorem \ref{thm.main.E'} corresponds to the limit case of Theorem \ref{thm.main.E} when $L\to1$. Recall that $L$ represents the size of the cube $K$ in which the holes are located. This limit is not singular if one allows the control zone $\Omega_c$ to have intersections with $K$. Actually, all the estimates in Lemmas \ref{lem.Corr1}--\ref{lem.Corr3} and Proposition \ref{prop.RoC.E} are independent of $0<L<1$. Hence, by reproducing the proof of Theorem \ref{thm.main.E}, we conclude the proof of Theorem \ref{thm.main.E'}. 
\end{proof}

Notice that the cases where $1 < \alpha < 3/2$, which corresponds to larger sizes of holes, are not treated by Theorem \ref{thm.main.E'}. For these cases, we have the following result. Define
\begin{equation}\label{def.pD}
    \mathfrak p_D 
    = \mathfrak p_D(\alpha,\beta)
    = \min\Big\{\frac{3-2\beta}{3},\, \frac{\alpha-1}{2},\, 3-\alpha,\, 6-2\alpha-2\beta\Big\}. 
\end{equation}

\begin{theorem}\label{thm.main.D}
Let $1<\alpha<3$ and choose $\beta>0$ so that $\beta<\min\{3/2,3-\alpha\}$, which ensures that $\mathfrak p_D$ in \eqref{def.pD} is positive. For given $0<\eta<1$, we define 
\begin{equation}\label{def.eps_D}
    \eps_0 = \frac{1}{\lfloor \eta^{-2/{\mathfrak p}_D} \rfloor+1}, 
\end{equation}
where $\lfloor\,\cdot\,\rfloor$ refers to the floor function. Then, for any $0 < \eps \le \eps_0$ of the form \eqref{def.eps'} and initial data $U^\eps_0\in L^2_\sigma(\Omega^\eps)$ satisfying
\begin{equation}\label{est1.thm.main.D}
    \|U^\eps_0\|_{L^2(\Omega^\eps)}
    \le
    \eps^{\mathfrak p_D - \beta}, 
\end{equation}
the Lagrangian controllability of \eqref{eq.NS} in the time $\eps^{\alpha+2\beta-3}$ holds: there exists a smooth control forcing $F^\eps$ supported in $\Omega_c^\eps$ with $\|F^\eps\|_{L^{\infty}_t H^s_x} \le C_0 \eps^{-2\beta}$ for any $s\in\N$ 
and a constant $C_0=C_0(s)>0$ depending on $s$
such that 
\begin{equation}\label{est2.thm.main.D}
   \bigcup_{t\in [0,\eps^{\alpha+2\beta-3}]} \Phi^\eps(t,P_0^\eps) \subset\Omega^\eps\setminus\overline{\Omega_{c}^\eps}, \qquad 
    \mathcal{L}(P_1^\eps \setminus \Phi^\eps(\eps^{\alpha+2\beta-3}, P_0^\eps)) 
    \le C \eta, 
\end{equation}
where $\mathcal{L}$ refers to the Lebesgue measure and $\Phi^\eps$ the flow for a corresponding Leray-Hopf weak solution $U^\eps$ to \eqref{eq.NS}. The constant $C>0$ is independent of $\eta$ and $\eps$. 
\end{theorem}

\begin{remark}\label{rem.thm.main.D}
\begin{enumerate}[(i)]
\item\label{item1.rem.thm.main.D}
The definition of the exponent $\mathfrak p_D$ in \eqref{def.pD} is based on the rate of convergence in Proposition \ref{prop.RoC.D} in Section \ref{sec.Homogenization.D}, which is due to \cite[Theorem 1.3]{Hof23}. However, the exponent suggested by \cite[Theorem 1.3]{Hof23} is slightly different and given by
\[
    \min\Big\{\frac{3-2\beta}{3},\, \frac{\alpha-1}{2},\, \frac{9-3\alpha}{2},\, 6-2\alpha-2\beta\Big\}.
\]
The difference is due to our correction of \cite[(3.16)]{Hof23}: $\eps^{9-3\alpha}$ is corrected to $\eps^{6-2\alpha}$.

\item\label{item2.rem.thm.main.D}
It is not clear to the authors whether the cases where $1 < \alpha < 3/2$ can be treated in partially perforated domains in Section \ref{sec.LC.E}, since the proof of Theorem \ref{thm.main.D} is based on the Poincaré inequality \eqref{est.Poincare} only valid in fully perforated domains. Indeed an energy estimate does not work, at least directly, in the absence of \eqref{est.Poincare}.

\item\label{item3.rem.thm.main.D}
As in Remark \ref{rem.thm.main.E} (\ref{item5.rem.thm.main.E}), one can ask ``the best choice" of $\beta$ when optimizing the trio of the cost (the size of control forcings), the time, and the size of initial data. 
\end{enumerate}
\end{remark}

    \subsubsection*{Outlined proof of Theorem \ref{thm.main.D}.}

We outline the proof of Theorem \ref{thm.main.D} according to the four steps introduced in the outlined proof of Theorem \ref{thm.main.E}. However, the scaling in Step 1 is just a replacement of \eqref{eq.UPF_E} with 
\begin{equation}\label{eq.UPF_D}
\begin{split}
    U^\eps(t,x) 
    &= \eps^{-(\alpha+2\beta-3)} u^\eps(\eps^{-(\alpha+2\beta-3)}t,x), \\
    P^\eps(t,x) 
    &= \eps^{-2\beta} p^\eps(\eps^{-(\alpha+2\beta-3)}t,x), \\
    F^\eps(t,x) 
    &= \eps^{-2\beta} f^\eps(\eps^{-(\alpha+2\beta-3)}t,x). 
\end{split}
\end{equation}
In addition, the homogenization in Step 2 just uses the same corrector and the energy estimates as in \cite{Hof23}. Thus we omit Steps 1 and 2, and describe Steps 3 and 4. 
\begin{enumerate}[Step 1.]
\setcounter{enumi}{2}
\item 
\textbf{Stability.} 
By a weak-strong stability argument for flows, we convert the convergence rate obtained in Step 2 between $u^\eps$ and $u_D$  with the estimate between $\phi^\eps$ and $\phi_D$, where $\phi_D$ refers to the flow of $u_D$ obtained in Step 4 below.

\item
\textbf{Limit equations.} 
We construct a smooth triplet $(u_D, p_D, f_D)$ solving the Darcy equations in $\T^3$ such that the initial zone $P_0$ moves by $u_D$ to the zone $P_1$, up to arbitrarily small error in the measure. Moreover, we take $f^{\eps}$ by $f^\eps := f_D|_{\Omega^\eps}$ and define the control forcing $F^\eps$ through \eqref{eq.UPF_D}. Our construction of $(u_D, p_D, f_D)$ is new. Unlike the proof of Theorem \ref{thm.main.E}, a simple adaptation of Glass-Horsin \cite{GlaHor12} does not suffice. We apply Runge type approximations for elliptic equations in \cite{Bro61} and Cauchy–Kowalevsky type theorems for equations with analytic coefficients.
\end{enumerate}

For the same reason as in Section \ref{sec.LC.E}, we do not pursue the Lagrangian controllability by an Euler-Brinkman homogenization in fully perforated domains.

    \paragraph{Notation}

In the rest of this section, we will adopt the same notation as in Section \ref{sec.LC.E}.

    \subsection{Controllability of Darcy equations in torus}
    \label{sec.Limit.D}

In this section, let $A = (A_{jk})\in\R^{3\times3}$ denote a positive definite symmetric matrix. We will prove the following remote controllability of the Darcy equations in $\T^3$.

\begin{proposition}\label{prop.CoD}
Assume that Jordan surfaces $\gamma_0$ and $\gamma_1$ are isotopic in $\T^3\setminus\overline{\Omega_{c}}$ and surround the same volume. For any $\eta>0$ and $k\in\N$, there exist smooth $(u_D, p_D, f_D)$ and $d_0>0$ satisfying 
\begin{equation}\label{eq1.prop.CoD}
    \left\{
    \begin{array}{ll}
    A u_D + \nabla p_D = f_D &\mbox{in}\ (0,1)\times\T^3,\\
    \nabla\cdot u_D = 0 &\mbox{in}\ (0,1)\times\T^3,\\
    u_D=0 &\mbox{on}\ \{0\}\times\T^3
    \end{array}\right.
\end{equation}
and the conditions 
\begin{gather}
    \opsupp f_D(t,\cdot) \subset \overline{\Omega_c}, 
    \quad t\in(0,1); \label{item1.prop.CoD} \\
    \opdist(\phi_D(t,\gamma_0),\Omega_c)\geq d_0, 
    \quad t\in [0,1]; \label{item2.prop.CoD} \\
    \|\phi_D(1,\gamma_0)-\gamma_1\|_{C^k(\S^2)}<\eta, \label{item3.prop.CoD}
\end{gather}
up to reparameterization. Here $\phi_D$ denotes the flow of $u_D$. 
\end{proposition}

In the proof, we need Lemma \ref{lem.PotFlows.D} below, whose proof is postponed for a moment.

\begin{lemma}\label{lem.PotFlows.D}
Assume that Jordan surfaces $\gamma_0$ and $\gamma_1$ are isotopic in $\T^3\setminus\overline{\Omega_c}$ and surround the same volume. For any $\eta>0$ and $k\in\N$, there exist smooth $\theta$ supported in $(0,1)\times\T^3$ and $d_0>0$ satisfying the conditions 
\begin{gather}
    \opsupp \nabla\cdot(A^{-1}\nabla\theta(t,\cdot)) \subset \overline{\Omega_c}, \quad t\in(0,1); \label{item1.lem.PotFlows.D} \\
    \opdist(\phi^{\nabla\theta}(t,\gamma_0),\Omega_c)\geq d_0, \quad t\in[0,1]; \label{item2.lem.PotFlows.D} \\
    \|\phi^{\nabla\theta}(1,\gamma_0)-\gamma_1\|_{C^k(\S^2)}<\eta, \label{item3.lem.PotFlows.D}
\end{gather}
up to reparameterization. Here $\phi^{\nabla \theta}$ denotes the flow of $\nabla\theta$. 
\end{lemma}

\begin{proofx}{Proposition \ref{prop.CoD}}
Let $\eta>0$ and $k\in\N$ be given. Take $\theta$ in Lemma \ref{lem.PotFlows.D} and define
\[
    v_D = A^{-1} \nabla \theta, 
    \qquad
    q_D = -\theta, 
    \qquad
    g_D = \nabla\cdot (A^{-1} \nabla \theta).
\]
Then, we see that $(v_D,q_D)$ satisfies the Darcy equations with inhomogeneous divergence
\begin{equation}\label{eq1.prf.prop.CoD}
    \left\{
    \begin{array}{ll}
    A v_D + \nabla q_D = 0 &\mbox{in}\ (0,1)\times\T^3,\\
    \nabla\cdot v_D = g_D &\mbox{in}\ (0,1)\times\T^3,\\
    v_D=0 &\mbox{on}\ \{0\}\times\T^3. 
    \end{array}\right.
\end{equation}
By the same reasoning as in the proof of Lemma \ref{prop.CoE}, we see that a Bogovskii operator $\opbog=\opbog[\,\cdot\,]$ is applicable to $g_D$ and $w_D:=\opbog[g_D]$ satisfies 
\begin{equation}\label{eq2.prf.prop.CoD}
    \opsupp w_D \subset \overline{\Omega_c},
    \qquad
    \nabla\cdot w_D = g_D.
\end{equation}

Now define $u_D = v_D - w_D$, $p_D = q_D$, and
\[
    f_D = -A w_D.
\]
Then $(u_D,p_D,f_D)$ satisfies the properties in Proposition \ref{prop.CoD} thanks to \eqref{eq1.prf.prop.CoD}--\eqref{eq2.prf.prop.CoD}.
\end{proofx}

The rest of this section focuses on proving Lemma \ref{lem.PotFlows.D}. We apply the lemma below.

\begin{lemma}\label{lem.Runge.D}
Let $K$ be a closed subset of $\T^3$ such that $K\cap \Omega_c=\emptyset$ and $\T^3\setminus K$ is connected. Then, for any distribution 
$\varphi$ satisfying $\nabla\cdot(A^{-1}\nabla\varphi) = 0$ in a neighborhood of $K$, $\eta>0$ and $k\in\N$, there exists smooth $\psi$ on $\T^3$ such that $\opsupp \nabla\cdot(A^{-1}\nabla\psi) \subset \overline{\Omega_c}$ and $\|\varphi-\psi\|_{C^{k}(K)}<\eta$.
\end{lemma}

\begin{proof}
According to Schauder estimates, $\varphi$ is smooth in the neighborhood of $K$. Fix an open subset $\Omega_1$ of $\T^3$ such that $\Omega_1 \Subset \Omega_c$. Let $\mathcal{K}:=\T^3\setminus \Omega_1$ be embedded in the Euclidean space $\R^4$ and take an open neighborhood $G$ in $\R^4$. Then $G\setminus \mathcal{K}$ has no components with compact closure in $G$. By the Runge theorem in Browder \cite[Theorem 3.22]{Bro61}, there exists a function $\tilde{\psi}$ defined on $\T^3\setminus \Omega_1$ such that $\nabla\cdot (A^{-1}\nabla\tilde{\psi})=0$ in $\T^3\setminus \Omega_1$ and $\|\varphi-\tilde{\psi}\|_{L^{\infty}(K)}<\eta$. By Schauder estimates applied to $\varphi-\tilde{\psi}$, it is in fact possible to improve the topology of approximation to $\|\varphi-\tilde{\psi}\|_{C^k(K)}<\eta$. Now take a smooth cut-off function $\chi$ on $\T^3$ such that $\chi=1$ on $\T^3\setminus \Omega_c$ and $\chi=0$ on $\Omega_1$. Then $\psi:=\chi \tilde{\psi}$ satisfies the desired properties.
\end{proof}

\begin{proofx}{Lemma \ref{lem.PotFlows.D}}
As in the proof of Lemma \ref{lem.PotFlows.E}, we can assume that $\gamma_0, \gamma_1$ are real analytic in $\T^3$ and that there exists a real analytic vector field $X\in C^0([0,1]; C^\omega(\T^3)^3)$ such that $\nabla\cdot X = 0$ and the flow $\phi^X$ for $X$ satisfies $\phi^X(1,\gamma_0) = \gamma_1$ and \eqref{est0.prf.lem.PotFlows.E} for some $d_0>0$. Then, as described in the proof of Lemma \ref{lem.PotFlows.E}, the combination of an approximation procedure adopting Whitney’s approximation theorem and Lemma \ref{lem.KryGlaHor} provides the conclusions of Lemma \ref{lem.PotFlows.D}. Notice that stability of flows can be verified by Gr\"{o}nwall's inequality even if $\Delta$ is replaced with $\nabla\cdot (A^{-1}\nabla)$ since $A$ is a positive definite symmetric matrix.

Let $\eta>0$ and $k\in\N$ be given and set $\gamma(t) = \phi^X(t,\gamma_0)$ for $t\in[0,1]$. Fix $t\in[0,1]$. In the same manner as in \cite[Proof of Lemma 2.4]{GlaHor12}, one can find a solution $\psi$ in $H^2$ of 
\[   
    \begin{split}
    \left\{
    \begin{array}{ll} 
    \nabla\cdot (A^{-1}\nabla \psi) = 0 &\mbox{in}\ \opint(\gamma(t)), \\ [2pt] 
    \displaystyle{\frac{\partial \psi}{\partial \nu}} = X \cdot \nu &\mbox{on}\ \gamma(t) 
    \end{array}\right.
    \end{split}
\]
such that $\psi$ can be analytically extended to a connected open neighborhood $U_t$ of $\gamma(t)$. This is essentially due to the Cauchy–Kowalevsky type theorem \cite[Thoerem 5.7.1']{Mor66book}. Then, by the unique continuation of $\nabla\cdot (A^{-1}\nabla \psi)$ on $U_t$, one concludes that $\nabla\cdot (A^{-1}\nabla \psi)$ vanish everywhere on $U_t$. The construction of $\psi$ above can be made uniform in time, by using the compactness of the time interval $[0,1]$, that is, there exists $\xi>0$  and 
\[
    \psi\in C
    \big(
    [0, 1]; C^\infty(\mathcal{V}_{\xi}[\opint(\gamma(t))];\R)
    \big)
\]
such that, for each $t\in[0,1]$, 
\[
    \begin{split}
    \left\{
    \begin{array}{ll} 
    \nabla\cdot (A^{-1}\nabla \psi) = 0 &\mbox{in}\ \mathcal{V}_\xi[\opint(\gamma(t))], \\ [2pt] 
    \displaystyle{\frac{\partial \psi}{\partial \nu}} = X \cdot \nu &\mbox{on}\ \gamma(t). 
    \end{array}\right.
    \end{split}
\]
Then the same argument as in the proof of Lemma \ref{lem.PotFlows.E}, but with an application of Lemma \ref{lem.Runge.D}, asserts the existence of $\theta$. This completes the proof of Lemma \ref{lem.PotFlows.D}. 
\end{proofx}

    \subsection{Homogenization to Darcy equations}
    \label{sec.Homogenization.D}

Let $(u_D, p_D, f_D)$ be the triplet satisfying the Darcy equations in Proposition \ref{prop.CoD} 
 with $A$ now being the resistance matrix $\mathcal{R}$ defined in \eqref{resis.mat}.
\[
    f^\eps = f_D|_{\Omega^\eps}.
\]

Then we consider the Navier-Stokes equations in $\Omega^{\eps}$ with viscosity $\eps^\beta$ and forcing $f^\eps$ 
\begin{equation}\label{eq.NS_fc_D}
    \left\{
    \begin{array}{ll}
    \eps^{6-2\alpha-2\beta}
    (\partial_t u^\eps + u^\eps\cdot\nabla u^\eps)
    - \eps^{3-\alpha}\Delta u^\eps + \nabla p^\eps
    = f^\eps &\mbox{in}\ (0,1)\times \Omega^\eps,\\
    \nabla\cdot u^\eps = 0 &\mbox{in}\ (0,1)\times\Omega^\eps,\\
    u^\eps = 0 &\mbox{on}\ (0,1)\times\partial\Omega^\eps,\\
    u^\eps = u^\eps_0 &\mbox{on}\ \{0\}\times\Omega^\eps. 
    \end{array}
    \right.
\end{equation}

We omit the proof of the following proposition for Leray-Hopf weak solutions of \eqref{eq.NS_fc_D},  since it is completely parallel to \cite[Theorem 1.3]{Hof23} based on the estimates for the corrector in \cite[Section 2]{Hof23} along with the Poincaré inequality for $\varphi\in H^1_0(\Omega^\eps)$ in \cite[Lemma 3.4.1]{All90b} 
\begin{equation}\label{est.Poincare}
    \|\varphi\|_{L^2(\Omega^\eps)} \lesssim \eps^{\frac{3-\alpha}{2}} \|\nabla \varphi\|_{L^2(\Omega^\eps)}.
\end{equation}

\begin{proposition}\label{prop.RoC.D}
For $0<\eps<1$, we have 
\begin{equation}\label{est1.prop.RoC.D}
    \|u^\eps-u_D\|_{L^2(0,1;L^2(\Omega^\eps))}
    \lesssim 
    \eps^{3-\alpha-\beta} \|u^\eps_0\|_{L^2(\Omega^\eps)} 
    + \eps^{\mathfrak p_D}, 
\end{equation}
where the exponent $\mathfrak p_D$ is defined in \eqref{def.pD}.
\end{proposition}

\begin{remark}\label{rem.prop.RoC.D}
As noted in Remark \ref{rem.thm.main.D} (\ref{item1.rem.thm.main.D}), the rate of convergence in Proposition \ref{prop.RoC.D} is slightly different from the one in \cite[Theorem 1.3]{Hof23} due to a correction of \cite[(3.16)]{Hof23}. 
\end{remark}

    \subsection{Proof of Theorem \ref{thm.main.D}}
    \label{sec.Scaling-Stability.D}

In this section, we prove Theorem \ref{thm.main.D}. Define $(U^\eps, P^\eps, F^\eps)$ by the scaling \eqref{eq.UPF_D}. Then $(U^\eps, P^\eps, F^\eps)$ satisfies the Navier-Stokes equations in $\Omega^{\eps}$ with viscosity $1$ and forcing $F^\eps$ 
\begin{equation}
    \left\{
    \begin{array}{ll}
    \partial_t U^\eps - \Delta U^\eps + U^\eps\cdot\nabla U^\eps + \nabla P^\eps 
    = F^\eps &\mbox{in}\ (0,\eps^{\alpha+2\beta-3})\times\Omega^\eps,\\
    \nabla\cdot U^\eps = 0 &\mbox{in}\ (0,\eps^{\alpha+2\beta-3})\times\Omega^\eps,\\
    U^\eps = 0 &\mbox{on}\ (0,\eps^{\alpha+2\beta-3})\times\partial\Omega^\eps,\\
    U^\eps = U^\eps_0 &\mbox{on}\ \{0\}\times\Omega^\eps.
    \end{array}\right.
\end{equation}

Let $\phi^\eps, \Phi^\eps$ denote the flows on $\Omega^\eps$ for $u^\eps, U^\eps$, respectively, and let $\phi_D$ denote the flow on $\T^3$ for $u_D$. Since $u^\eps(t,x)=\eps^{\alpha+2\beta-3} U^\eps(\eps^{\alpha+2\beta-3} t,x)$ by \eqref{eq.UPF_D}, we have 
\begin{equation}\label{eq.FlowMaps_D}
    \phi^\eps(t,x) = \Phi^\eps(\eps^{\alpha+2\beta-3} t,x), 
    \quad
    t\in [0,1]. 
\end{equation}

We first prove the following lemmas for $\phi^\eps, \phi_D$. Recall that $\mathfrak p_D$ is defined in \eqref{def.pD}.

\begin{lemma}\label{lem.RoC.phiD}
For $0<\eps<1$, we have
\begin{equation}\label{est1.lem.RoC.phiD}
\begin{split}
    \|\phi^\eps(t) - \phi_D(t)\|_{L^2(\Omega^\eps)}
    \lesssim
    \eps^\beta \|U^\eps_0\|_{L^2(\Omega^\eps)} + \eps^{\mathfrak p_D},
    \quad
    t\in [0,1]. 
\end{split}
\end{equation}
\end{lemma}

\begin{proof}
The same argument as in the proof of Lemma \ref{lem.RoC.phiE} leads to, for $t\in[0,1]$, 
\begin{equation*}
    \|\phi^\eps(t) - \phi_D(t)\|_{L^2(\Omega^\eps)} 
    \lesssim
    \int_0^t 
    \|u^\eps(s) - u_D(s)\|_{L^2(\Omega^\eps)}
    \dd s.
\end{equation*}
Hence, with 
\[
     u^\eps_0=\eps^{\alpha+2\beta-3} U^\eps_0, 
\]
the assertion follows from the Cauchy-Schwarz inequality and Proposition \ref{prop.RoC.D}.
\end{proof}

\begin{lemma}\label{lem.Markov.D}
For given $0<\eta<1$, take $\eps_0$ defined in \eqref{def.eps_D}. 
Then, for any $0 < \eps \le \eps_0$ of the form \eqref{def.eps'} and $U^\eps_0$ satisfying $\|U^\eps_0\|_{L^2(\Omega^\eps)} \le \eps^{\mathfrak p_D - \beta}$, we have
\[
    \begin{split}
    \mathcal{L}\big(
    \{x\in \Omega^\eps~|~|\phi^\eps(t,x) - \phi_D(t,x)| > \eta\}
    \big)
    \lesssim 
    \eta,
    \quad
    t\in [0,1].
    \end{split}
\]
\end{lemma}

\begin{proof}
The same argument as in the proof of Lemma \ref{lem.Markov.E} leads to, for $t\in[0,1]$,
\begin{equation*}
    \begin{split}
    &\mathcal{L}\big(\{x\in \Omega^\eps~|~|\phi^\eps(t,x) - \phi_D(t,x)|>\eta\}\big) \\
    &\lesssim
    \eta^{-1} 
    \big(
    \eps^\beta \|U^\eps_0\|_{L^2(\Omega^\eps)} + \eps^{\mathfrak p_D}
    \big)
    \lesssim
    \eta^{-1} \eps^{\mathfrak p_D}.
    \end{split}
\end{equation*}
The assertion follows from the definition of $\eps_0$.
\end{proof}

With the previous materials in hand, we can now conclude the proof of Theorem \ref{thm.main.D} following the same lines as in the proof of Theorem \ref{thm.main.E}, see Section \ref{sec.Scaling-Stability.E}.

    \appendix

    \section{Definitions from differential geometry}
    \label{appx.Defs}

In this appendix, we recall the definitions from differential geometry needed in this paper. We follow the presentation in \cite[Section 1.2]{GlaHor12} where the main reference is \cite{Hir76book}. Hereafter, unless otherwise stated, all the geometrical objects will be considered in the smooth category.

We say that $\gamma$ is a Jordan surface if $\gamma$ embedded in $\R^3$ is the image of the $2$-sphere $\S^2$ by a smooth embedding $h: \S^2 \rightarrow \R^3$. Thanks to the Jordan–Brouwer theorem \cite{Bro11,GuiPol74}, the complement $\R^3\setminus \gamma$ has two connected components, one of which is bounded in $\R^3$ and denoted by $\opint[\gamma]$. The unit outward normal vector field on $\opint[\gamma]$ is denoted by $\nu$.

Let $\Omega$ be a connected set of $\R^3$. Two Jordan surfaces $\gamma_0$ and $\gamma_1$ embedded in $\R^3$ are said to be isotopic in $\Omega$ if there exists a continuous mapping $I: [0,1]\times \S^2 \rightarrow \Omega$, called an isotopy, such that $I(0)=\gamma_0, I(1)=\gamma_1$ and that $I(t,\cdot)$ is a homeomorphism of $\S^2$ into its image for each $t\in[0,1]$. We say that $I$ is a smooth isotopy if the homeomorphism $I(t,\cdot)$ is a $C^\infty$-diffeomorphism with respect to the space variable for each $t\in[0,1]$. Finally, an one-parameter continuous family of diffeomorphism of $\Omega$ is called diffeotopy of $\Omega$.

These notions are naturally modified to their versions in the $3$-torus $\T^3$. Recall that $\T^3$ is equipped with a Riemannian metric $g$ for which the quotient mapping $\pi: \R^3 \to \T^3$ is a local isometry. Then a Jordan surface $\gamma$ in $\T^3$ can be defined by starting with a Jordan surface $\gamma$ embedded in $\R^3$ and contained in $[0,1)^3$. The meaning of $\opint[\gamma]$ for a Jordan surface $\gamma$ in $\T^3$ should be clear in this context. To avoid redundancy, we omit more details for the definitions of geometrical objects in $\T^3$ such as isotopicity of Jordan surfaces in $\T^3$.

    \section{Flows on a perforated domain}
    \label{appx.Flow}

Let $\Omega^\eps$ be a partially or fully perforated domain defined in Section \ref{sec.LC.E} or \ref{sec.LC.D}, respectively. Let $T > 0$ and a sufficiently regular vector field $u: [0,T]\times \Omega^\eps\rightarrow \R^3$ satisfy $\nabla\cdot u=0$ in the sense of distributions. Then, reformulating the definition in Glass-Horsin \cite[Section 1.2]{GlaHor12} for smooth and bounded domains $\Omega\subset\R^3$, we say that a mapping $\phi^u: [0,T]\times \Omega^\eps \rightarrow \R^3$ is a flow for $u$ if $\phi^u$ is a solution of the following integral equation for a.e. $x\in \Omega^\eps$: 
\begin{equation}\label{eq.integral}
    \phi(t,x) = x + \int_{0}^{t} u(s,\phi(s,x)) \dd s, \quad t\in [0,T]. 
\end{equation}
For our application, we are interested in the existence and uniqueness of flows $\phi^u$ for a Leray-Hopf weak solution $u$ of the Navier-Stokes equations with initial data in $L^2_\sigma(\Omega^\eps)$. 
Indeed, the existence of $\phi^u$ in the class $C_t L^2_x$ is provided by Foia\c{s}-Guillop\'e-Temam \cite[Theorem 2.1]{FGT85} using the a priori estimates of $u$ in \cite{FGT81}, and moreover, these flows $\phi^u$ can be assumed to be volume-preserving by \cite[Theorem 3.1]{FGT85}.

The uniqueness of $\phi^u$ on $\R^3$ in the class $C_t L^2_x$ in the sense of path-by-path is recently provided by Galeati \cite{Gal25} who revisits earlier works for more regular initial data by Robinson-Sadowski \cite{RobSad2009b,RobSad2009a}. The proof of \cite{Gal25} is done in the framework of regular Lagrangian flows \cite{DiPLio89,Amb04,CriDel08,BCD21,Inv23}. Note that the uniqueness by \cite{Gal25} is different from the uniqueness of regular Lagrangian flows, e.g., Crippa-De Lellis \cite[Corollary 3.7]{CriDel08}. The notion of the latter involves a global constraint due to the definition; see \cite{BCD21}. In our problems, the path-by-path uniqueness $\phi^u$ on $\Omega^\eps$ in the class $C_t L^2_x$ can be proved by following the ideas in \cite{BCD21,Gal25}. In fact, the proof is just a reproduction of \cite[Proof of Corollary 10.1]{BCD21} based on the asymmetric Lusin–Lipschitz estimate of the form \cite[Corollary 4.3]{Gal25} for Leray-Hopf weak solutions in $\Omega^\eps$. Therefore, we omit the details here to avoid repetition.

\subsection*{Acknowledgements}
The authors would like to thank Richard M. H\"{o}fer for helpful comments. 
MH was supported by JSPS KAKENHI Grant Numbers JP 25K17278 and 25K00915. 
JL was supported by National Natural Science Foundation of China under Grant No. 12301238. 
FS was supported by the project ANR-23-CE40-0014-01 BOURGEONS of the French National Research Agency (ANR).

    \addcontentsline{toc}{section}{References}
    \bibliography{Ref}
    \bibliographystyle{plain}

\end{document}